\documentclass[reqno, 12pt]{amsart}

\usepackage{_pascal/_pascal}

\hypersetup{
    pdftitle={Tree-independence number of P5-free graphs with no large bicliques},
}

\newcommand{\tin}{\mathsf{tree}\textnormal{-} \alpha}
\newcommand{\NN}{N_{\overline{r}}}
\newcommand{\cc}{\ensuremath{\mathcal{C}}}

\begin{document}

\author[Blažej]{Václav Blažej$^\ast$}
\address[$\ast$]{Institute of Informatics, University of Warsaw, Warsaw, Poland}
\email{\tt v.blazej@uw.edu.pl}

\author[Gollin]{J.~Pascal Gollin$^\dagger$}
\address[$\dagger$]{FAMNIT, University of Primorska, Koper, Slovenia}
\email{\tt pascal.gollin@famnit.upr.si}

\author[Hons]{Tomáš Hons$^\S$}
\address[$\S$]{Computer Science Institute, Charles University, Czech Republic and Institute of Computer Science, Czech Academy of Sciences, Czech Republic}
\email{\tt honst@iuuk.mff.cuni.cz}

\author[Masařík]{Tomáš Masařík$^\P$}
\address[$\P$]{Institute of Informatics, University of Warsaw, Warsaw, Poland}
\email{\tt masarik@mimuw.edu.pl}

\author[Milanič]{Martin Milanič$^{\dagger,\ddagger}$}
\address[$\ddagger$]{IAM, University of Primorska, Koper, Slovenia}
\email{\tt martin.milanic@upr.si}

\author[Rzążewski]{Paweł Rzążewski$^\parallel$}
\address[$\parallel$]{Warsaw University of Technology, Warsaw, Poland}
\email{\tt pawel.rzazewski@pw.edu.pl}

\author[Suchý]{Ondřej Suchý$^{\#}$}
\address[$\#$]{Faculty of Information Technology, Czech Technical University in Prague, Prague, Czech Republic}
\email{\tt ondrej.suchy@fit.cvut.cz}

\author[Wesolek]{Alexandra Wesolek$^\top$}
\address[$\top$]{CNRS, LaBRI, University of Bordeaux, Bordeaux, France}
\email{\tt alexandra.wesolek@u-bordeaux.fr}

\title[Tree-independence number of $P_5$-free graphs with no large bicliques]{Tree-independence number of $P_5$-free graphs with no large bicliques}

\date{September 17, 2026}

\begin{abstract}
    The tree-independence number of a graph is the minimum, over all tree-decompositions of the graph,
    of the maximum size of an independent set contained in a bag.
    Graph classes of bounded tree-independence number have strong structural and algorithmic properties;
    however, the parameter can be unbounded even in quite restricted classes.
    In particular, the presence of an induced biclique~$K_{\ell,\ell}$ forces tree-independence number at least~$\ell$.
    This leads to the question whether large induced bicliques are the only obstruction to bounded tree-independence number in natural hereditary classes.
    A conjecture of Dallard, Krnc, Kwon, Milani\v{c}, Munaro, \v{S}torgel, and Wiederrecht states that for all positive integers~$t$ and $\ell$, ${\{P_t,K_{\ell,\ell}\}}$-free graphs have bounded tree-independence number. 
    We prove this conjecture for~${t=5}$ by showing that every ${\{P_5,K_{\ell,\ell}\}}$-free graph has tree-independence number at most~${4\ell-4}$.
    We also obtain related bounds for the weaker parameter of $\alpha$-degeneracy and answer a question of Hilaire, Milanič, and Vasić whether tree-independence number of ${\{P_5,K_{\ell,\ell}\}}$-free graphs exceeds~$\ell$ by at most an additive constant. 
\end{abstract}

\maketitle

\blindfootnote{An extended abstract of this work has been published at the 34th European Symposium on Algorithms (ESA 2026); see \cite{blazej2026treeindependencenumberp5freegraphs}.}

\section{Introduction}
\label{sec:intro}

For many graph parameters, one can identify natural substructures whose presence forces the parameter to be large. 
A basic question is whether, inside a hereditary\footnote{A graph class is \emph{hereditary} if it is closed under taking induced subgraphs.} 
graph class, these obvious obstructions are the only reason for large parameter values. 
A classical example is $\chi$-boundedness: which hereditary graph classes have chromatic number bounded in terms of clique number? 
This notion, introduced by Gy{\'a}rf{\'a}s~\cite{MR951359}, has been extensively studied; see, for instance, the survey~\cite{MR4174126}.

Other examples include bounding treewidth in terms of clique number (see~\cite{MR4332111,MR4334541}),  bounding pathwidth in terms of clique number (see~\cite{MR3215457,hajebi2025polynomialboundspathwidth}), bounding degeneracy in terms of clique number (see~\cite{MR3794363}), bounding minimum degree in terms of chromatic number (see~\cite{MR2811077}), bounding minimum degree in terms of the largest size of a balanced biclique (see~\cite{MR4931191}), and bounding coloring number in terms of chromatic number (see~\cite{MR1304254}). 

In this paper we study an analogous question for \emph{tree-independence number} and \emph{induced biclique number}. 
Tree-independence number was introduced independently by Yolov~\cite{MR3775804} (under the name \emph{$\alpha$-treewidth}) and by Dallard, Milani{\v c}, and {\v S}torgel~\cite{MR4664382}; see~\zcref{sec:prelim} for definitions. 
It is a relaxation of treewidth that can be bounded even on dense graph classes; at the same time, bounded tree-independence number still yields strong structural and algorithmic consequences (see~\cite{MR4664382,LMMORS24}). 
The \emph{induced biclique number} of a graph~$G$ is the maximum integer~$\ell$ such that~$G$ contains an induced subgraph isomorphic to~$K_{\ell,\ell}$. 
Since induced biclique number is a lower bound on tree-independence number~\cite{MR4664382}, it is natural to ask the following. 

\begin{question}
    \label{question:main}
    For which graph classes is tree-independence number bounded from above in terms of induced biclique number?
\end{question}

Before we discuss this question, let us introduce some terminology. 
Given a graph~$G$ and a set~$\mathcal{F}$ of graphs, we say that~$G$ is \emph{$\mathcal{F}$-free} if no induced subgraph of~$G$ is isomorphic to a member of~$\mathcal{F}$; 
for a graph~$H$, we write \emph{$H$-free} for \emph{$\{H\}$-free}. 
We denote the $t$-vertex path by~$P_t$, the disjoint union of two graphs~$H$ and~$H'$ by~${H+H'}$, 
and the disjoint union of~$t$ copies of a graph~$H$ by~$tH$.

\zcref{question:main} is meaningful already for quite restricted classes: 
for instance, grid graphs have induced biclique number at most~$2$, but unbounded tree-independence number~\cite{MR4730297}. 
On the positive side, tree-independence number is known to be bounded in terms of induced biclique number for $P_4$-free graphs~\cite{DKKMMSW2024}, ${(P_3+P_1)}$-free graphs~\cite{HMV25}, graph classes of bounded \textsl{induced matching treewidth}~\cite{MR4906164}, and consequently also $tK_2$-free graphs, for every positive integer~$t$.
These results motivate the following refinement of \zcref{question:main}.

\begin{question}
    \label{question:H-tree-alpha}
    For which graphs~$H$ is tree-independence number of $H$-free graphs bounded from above in terms of induced biclique number?
\end{question}

A related question, asking for which graphs~$H$ the class of $H$-free graphs has bounded tree-independence number, was answered by Dallard, Milanič, and Štorgel~\cite{MR4730297}: 
this happens if and only if~$H$ is edgeless or an induced subgraph of~$P_3$. 
Moreover, \zcref{question:H-tree-alpha} is equivalent to asking for which graphs~$H$ and every positive integer~$\ell$ the class of ${\{H,K_{\ell,\ell}\}}$-free graphs has bounded tree-independence number. 
A necessary condition is that~$H$ is a disjoint union of paths, by a result of Lozin and Razgon~\cite{MR4385180} on hereditary classes of bounded treewidth defined by finitely many forbidden induced subgraphs; see also~\cite{DKKMMSW2024}.

It is open whether this necessary condition is also sufficient. 
Since every disjoint union of paths is an induced subgraph of a path, a positive answer would follow from the following conjecture. 

\begin{conjecture}[Dallard, Krnc, Kwon, Milanič, Munaro, Štorgel, and Wiederrecht~\cite{DKKMMSW2024}]
    \label{conj:DKKMMSW2024}
    For every two positive integers~$t$ and~$\ell$, the class of ${\{P_t,K_{\ell,\ell}\}}$-free graphs has bounded tree-independence number.
\end{conjecture}

This conjecture is known to hold for~${t \leq 4}$ and every~$\ell$~\cite{DKKMMSW2024}. 
On the other hand, it is open already for~${\ell = 2}$ and every~${t \geq 7}$~\cite{chudnovsky2026treeindependencenumberforbiddeninduced}, while the case~${\ell = 1}$ is trivial. 
One can also ask for analogous statements with an unbalanced biclique. 
For example, boundedness of tree-independence number is known to hold for ${\{P_t,K_{1,\ell}\}}$-free graphs~\cite{DKKMMSW2024} (see also~\cite{ChoiHMW25}), whereas the case of ${\{P_t,K_{2,\ell}\}}$-free graphs is already open for all~${t \geq 7}$ and~${\ell \geq 2}$; it was only recently settled for~${t \in \{5,6\}}$ and arbitrary~$\ell$~\cite{chudnovsky2026treeindependencenumberforbiddeninduced}.

\paragraph{Our focus}
We study the first open case of \zcref{conj:DKKMMSW2024} with respect to~$t$, namely~${t=5}$. 
This is a natural next step in the context of \zcref{question:H-tree-alpha}, because~$P_5$ contains each graph in ${\{P_4,P_3+P_1,2P_2\}}$ as an induced subgraph. 

At first glance, one might expect $P_5$-free graphs to behave like a mild extension of $P_4$-free graphs, that is, cographs. 
In reality, the class is considerably richer. 
For instance, the complexity of \textsc{Independent Set} on $P_5$-free graphs remained open for more than twenty years, and its eventual polynomial-time solution required substantial new machinery~\cite{LokshantovVV14}. 
Likewise, both the Erd\H{o}s--Hajnal property and polynomial $\chi$-boundedness for $P_5$-free graphs were established only very recently~\cite{MR5049499,nguyen2025polynomialchiboundednessexcludingp5}, after a long sequence of partial results~\cite{BonamyBPRTW2022,MR3010736,MR4174126,MR3150572,MR634555,MR4648583}. 
These examples suggest that the class of $P_5$-free graphs is structurally rich but still sufficiently tame to admit strong decomposition theorems and algorithmic consequences.

Another indication of this intermediate behavior is that $P_5$-free graphs of bounded clique number have bounded \textsl{mim-width}~\cite{MR4293026}, a parameter that is very useful algorithmically~\cite{BergougnouxDJ23}. 
From this perspective, \zcref{conj:DKKMMSW2024} can be viewed as a biclique analogue of such boundedness phenomena. 
It is also worth noting that if one excludes both a clique and a biclique, then $P_t$-free graphs, for every fixed~$t$, have bounded treewidth~\cite{AtminasLR12}.

\paragraph{Our results}
Our main theorem confirms the case~${t=5}$ of \zcref{conj:DKKMMSW2024} and answers \zcref{question:H-tree-alpha} for the graph~$P_5$.

\begin{restatable}{theorem}{thmtreealpha}
    \label{thm:P5treealpha}
    For every integer~${\ell \geq 2}$, every $P_5$-free graph with no induced~$K_{\ell,\ell}$ has tree-independence number at most~${4\ell-4}$.
\end{restatable}

In particular, tree-independence number of $P_5$-free graphs is bounded in terms of induced biclique number. 
Actually, the proof of \zcref{thm:P5treealpha} is algorithmic in the following strong sense: 
given a $P_5$-free graph and an integer~$\ell$, in polynomial time we can either find an induced~$K_{\ell,\ell}$ or construct a tree-decomposition of independence number at most~${4\ell-4}$. 
In particular, the complexity of the algorithm depends only polynomially on~$\ell$.

We prove \zcref{thm:P5treealpha} by induction on the number of vertices. 
The key intermediate result is a bound on a weaker parameter, which is of independent interest.

For a graph~$G$, the \emph{$\alpha$-degeneracy} of~$G$ is the smallest integer~$k$ such that every non-null induced subgraph~$H$ of~$G$ contains a vertex $v$ with~${\alpha(H[N[v]]) \leq k}$. 
This parameter was considered previously in the literature~\cite{MR2916349}; the definition there uses open neighborhoods, but for our purposes the closed-neighborhood version is more convenient, and the two variants coincide on every graph with at least one edge. 
It is easy to see that $\alpha$-degeneracy is a lower bound on tree-independence number; see \zcref{lem:tree-alpha-deg}. 
Thus, bounded $\alpha$-degeneracy is a necessary condition for bounded tree-independence number, and it is natural to ask for a corresponding analogue of \zcref{conj:DKKMMSW2024}.

\begin{conjecture}
    \label{conj:alpha-degeneracy}
    For every two positive integers~$t$ and~$\ell$, the class of ${\{P_t,K_{\ell,\ell}\}}$-free graphs has bounded $\alpha$-degeneracy.
\end{conjecture}

It follows from the aforementioned results on variants of \zcref{conj:DKKMMSW2024} that the analogues of \zcref{conj:alpha-degeneracy} with~$K_{\ell,\ell}$ replaced by~$K_{1,\ell}$ are known for all~$t$, and with~$K_{\ell,\ell}$ replaced by~$K_{2,\ell}$ are known for~${t \leq 6}$; 
the latter remains open for all~${\ell \geq 2}$ and~${t \geq 7}$.

The parameter $\alpha$-degeneracy leads to a companion version of \zcref{question:H-tree-alpha}. 
Indeed, every vertex in a balanced biclique~$K_{\ell,\ell}$ has an independent neighborhood of size~$\ell$, so induced biclique number is a lower bound not only on tree-independence number, but also on $\alpha$-degeneracy.

\begin{question}
    \label{question:H-alpha-degeneracy}
    For which graphs~$H$ is $\alpha$-degeneracy of $H$-free graphs bounded from above in terms of induced biclique number?
\end{question}

Equivalently, \zcref{question:H-alpha-degeneracy} asks for which graphs~$H$ and every positive integer~$\ell$ the class of ${\{H,K_{\ell,\ell}\}}$-free graphs has bounded $\alpha$-degeneracy. 
A necessary condition is that~$H$ is a forest, by combining a theorem of Kierstead and Penrice~\cite{MR1258244} on hereditary classes of bounded degeneracy with Ramsey's theorem; see also~\cite[Theorem 22]{MR4508171}.

Our first ingredient for the proof of \zcref{thm:P5treealpha} is the special case of \zcref{conj:alpha-degeneracy} for~${t=5}$.

\begin{theorem}
    \label{thm:alpha-degeneracy-P5-Kll-simplified}
    For every $\ell \geq 2$, every non-null $\{P_5,K_{\ell,\ell}\}$-free graph has 
    $\alpha$-degeneracy at most $2\ell-2$. 
\end{theorem}

This is precisely the vertex deleted in the inductive proof of \zcref{thm:P5treealpha}. 
Our approach in fact works in a more general setting. 
Let~$S_d$ denote the graph obtained from the $d$-leaf star~$K_{1,d}$ by subdividing each edge once, so that~${S_2=P_5}$.

\begin{theorem}
    \label{thm:alpha-degeneracy-Sd-Kll-simplified}
    For every two integers~${d,\ell \geq 2}$, every non-null ${\{S_d,K_{\ell,\ell}\}}$-free graph has
    $\alpha$-degeneracy less than~${d^2\ell + 2d\ell^{d-1}}$. 
\end{theorem}

In particular, \zcref{thm:alpha-degeneracy-Sd-Kll-simplified} answers \zcref{question:H-alpha-degeneracy} whenever~$H$ is an induced subgraph of~$S_d$ for some~${d \geq 2}$.

We construct examples showing that the bound given by \zcref{thm:alpha-degeneracy-P5-Kll-simplified} is best possible and the bound given by \zcref{thm:P5treealpha} cannot be improved by more than a factor of~$2$.
This also answers a question of Hilaire, Milanič, and Vasić from~\cite{HMV25}, asking whether there exists a constant~$c$ such that the tree-independence number of a $P_5$-free graph is bounded from above by the sum of its induced biclique number and~$c$. 
In fact, the examples show that such an additive constant does not exist even for $2K_2$-free graphs, and that no multiplicative constant strictly less than~$2$ can exist. 

\begin{restatable}{proposition}{proplb}
    \label{prop:lower-bound}
    For every~${\ell \geq 2}$, there exists a ${\{2K_2,K_{\ell,\ell}\}}$-free graph with \linebreak $\alpha$\dash{}degeneracy and tree-independence number equal to~${2\ell-2}$. 
\end{restatable}

We also consider a complementary direction in which~$t$ is arbitrary but the forbidden biclique is fixed and unbalanced.
In particular, we show that for every~${\ell \geq 2}$, the $\alpha$-degeneracy of an $n$-vertex ${\{P_t,K_{2,\ell}\}}$-free graph is bounded by ${(t-1)\ell \cdot \log n}$. 
Actually, in the proof of this result we only use the fact that $P_t$-free graphs admit dominated balanced separators~\cite{MR3910075}. 
We say that a hereditary graph class~$\mathcal{G}$ has \emph{$d$-dominated balanced separators} if every graph~${G \in \mathcal{G}}$ contains a set~${X \subseteq V(G)}$ of size at most~$d$ such that every component of~${G-N[X]}$ has at most~${\frac{1}{2} \abs{V(G)}}$ vertices. 
We prove the following more general statement.

\begin{restatable}{theorem}{thmdbs}
    \label{thm:dbs}
    Let~${d \geq 2}$ be an integer and let~$\mathcal{G}$ be a class with $d$-dominated balanced separators. 
    Then, for every integer~${\ell \geq 2}$, every $K_{2,\ell}$-free graph from~$\mathcal{G}$ with~${n \geq 2}$ vertices has $\alpha$-degeneracy at most~${d\ell \log n}$. 
\end{restatable}

Since the class of $P_t$-free graphs admits ${(t-1)}$-dominated balanced separators~\cite{MR3910075}, \zcref{thm:dbs} immediately implies the aforementioned statement for ${\{P_t,K_{2,\ell}\}}$-free graphs. However, we know more classes with $d$-dominated balanced separators, and \zcref{thm:dbs} applies to all of them.

Gartland and Lokshtanov conjectured (see~\cite{GartlandThesis}) that for every planar graph~$H$, there exists an integer~$d$ such that graphs excluding~$H$ as an \textsl{induced minor} have $d$-dominated balanced separators.
This conjecture is confirmed for some restricted graphs~$H$, including paths~\cite{MR951359,MR3910075}, cycles~\cite{MR4708895}, and, more generally, wheels~\cite{chudnovsky2025}, as well as for some other cases (see~\cite{AhnGHK25,chudnovsky2025treeindependencenumberv,MR4730297}). 

For a positive integer~$t$, we denote by~$S_{t,t,t}$ the graph obtained from the complete bipartite graph~$K_{1,3}$ by subdividing each edge~${t-1}$ times, and by~$W_{t\times t}$ the $t$-by-$t$ hexagonal grid (also known as the ${t \times t}$-wall). 
Chudnovsky, Codsi, Lokshtanov, Milanič, and Sivashankar~\cite{chudnovsky2025treeindependencenumberv} showed that for every positive integer~$t$ there is an integer~$c_t$ such that if~$G$ is a ${\{K_{t,t}, S_{t,t,t}\}}$-free graph with~${n \geq 2}$ vertices that also excludes, as an induced subgraph, the line graphs of all subdivisions of~$W_{t\times t}$, then the tree-independence number of~$G$ is at most~${c_t\log^4n}$. 
This implies that every such graph has $\alpha$-degeneracy bounded by~${c_t\log^4 n}$. 

Chudnovsky et al.~also showed that for every positive integer~$t$, there exists an integer~${d_t \geq 2}$ such that the class of $S_{t,t,t}$-free graphs excluding the line graphs of all subdivisions of~$W_{t\times t}$ is a class with $d_t$-dominated balanced separators~\cite{chudnovsky2025treeindependencenumberv}. 
This result, combined with \zcref{thm:dbs}, shows that every ${\{K_{2,\ell},S_{t,t,t}\}}$-free graph with~${n \geq 2}$ vertices that also excludes, as an induced subgraph, the line graphs of all subdivisions of~$W_{t\times t}$, has $\alpha$-degeneracy bounded by~${d_t\ell\log n}$. 

\paragraph{Progress since this work} 
Let us briefly summarize the progress in the field following the submission of the main results of this work to ESA 2026~\cite{blazej2026treeindependencenumberp5freegraphs}.
Very recently and independently of this work, Hajebi and Spirkl~\cite{HajebiSpirkl2026} established the validity of \zcref{conj:DKKMMSW2024} (and consequently of \zcref{conj:alpha-degeneracy}) in full. 
In fact, they proved a stronger result implying in particular that for any finite set~$\mathcal{F}$ of graphs, the class of $\mathcal{F}$-free graphs has bounded tree-independence number if and only if~$\mathcal{F}$ contains a complete bipartite graph, a graph each component of which is a path or a subdivided claw, and the line graph of such a graph. 
While their approach is very general, they do not give explicit estimates on the obtained bounds on tree-independence number in terms of induced biclique number. 
They use Ramsey's theorem and similar tools that result in bounds that are at least exponential in the induced biclique number. 
On the other hand, our approach for \zcref{thm:P5treealpha} is tailored to the class of $P_5$-free graphs and results in bounds that are at most a constant multiplicative factor away from the optimal bounds.

\section{Preliminaries}
\label{sec:prelim}

All graphs considered in this paper are finite, simple, and undirected.
A graph is said to be \emph{null} if it has no vertices.
Let $G$ be a graph. 
For a vertex $v\in V(G)$, the \emph{(open) neighborhood} of $v$ in $G$ is the set $N_G(v)$ of vertices adjacent to $v$ in $G$; the set $N_G(v)\cup \{v\}$, denoted by $N_G[v]$, is the \emph{closed neighborhood} of $v$ in $G$.
For a set $X\subseteq V(G)$, we denote by $N_G[X]$ the set $\bigcup_{v\in X}N_G[v]$ and by $N_G(X)$ the set $N_G[X]\setminus X$.
The \emph{degree} of a vertex $v\in V(G)$ is denoted by $\deg_G(v)$ and defined as the cardinality of $N_G(v)$.
In all these cases, we omit the index $G$ whenever the graph is clear from the context.
For $X \subseteq V(G)$, by $G[X]$ we denote the subgraph of $G$ induced by the set $X$.
Since we mainly work with induced subgraphs, we often identify an induced subgraph (in particular, a component) with its vertex set, e.g., we write that $X$ is a connected set, meaning that $G[X]$ is connected.

Given a graph $G$, we say that $X \subseteq V(G)$ \emph{is complete to} $Y \subseteq V(G)$ if every vertex in $X$ is adjacent to every vertex of $Y$.
A \emph{clique} in a graph $G$ is a set of pairwise adjacent vertices; an \emph{independent set} in $G$ is a set of pairwise nonadjacent vertices.
The \emph{independence number} of a graph $G$, denoted by $\alpha(G)$, is the maximum size of an independent set in $G$.
For a positive integer $t$, we denote by $K_t$ the complete graph with $t$ vertices.
A \emph{matching} in a graph $G$ is a set of pairwise disjoint edges.
An \emph{induced matching} in a graph $G$ is a matching~$M$ such that no edge of the graph joins two distinct edges in $M$. 
A graph $G$ is \emph{bipartite} if it admits a \emph{bipartition}, that is, a pair of disjoint independent sets with union $V(G)$.
For two positive integers $m$ and $n$, the \emph{complete bipartite graph} $K_{m,n}$ is a bipartite graph that admits a bipartition into parts of size $m$ and $n$, respectively, that are complete to each other. 
A \emph{vertex cover} in a graph $G$ is a set $S$ of vertices of $G$ such that $V(G)\setminus S$ is an independent set.
By K\"onig's theorem, in every bipartite graph, the maximum size of a matching equals the minimum size of a vertex cover.

Let $G$ be a graph and let $\mathcal{T} = (T,\{B_t\}_{t \in V(T)})$ such that $T$ is a tree and $\{B_t\}_{t \in V(T)}$ is a collection of \emph{bags}, that is, subsets of $V(G)$ indexed by the nodes of $T$.
For a vertex $v \in V(G)$, denote by $T(v)$ the subgraph of $T$ induced by the nodes $t$ of $T$ such that the bag $B_t$ contains $v$.
We say that $\mathcal{T}$ is a \emph{tree-decomposition} of $G$ if the following conditions are satisfied: (i) for each vertex $v\in V(G)$, the set $T(v)$ is nonempty and connected, and (ii) for edge $uv \in E(G)$ it holds that $T(u) \cap T(v) \neq \emptyset$.
The \emph{independence number} of a tree-decomposition $\mathcal{T} = (T,\{B_t\}_{t \in V(T)})$ of a graph $G$, denoted by $\alpha(\mathcal{T})$, is the maximum, over all nodes $t$ of $T$, of the independence number of the subgraph of $G$ induced by the bag~$B_t$.
The \emph{tree-independence number} of~$G$, denoted by $\tin(G)$, is the minimum independence number of a tree-decomposition of~$G$.

We recall the following well-known property of tree-decompositions (see, e.g., \cite[Lemma 2.6]{MR4664382}).

\begin{lemma}
    \label{lem:closed-neighborhoods}
    Let $\mathcal{T} = (T,\{B_t\}_{t \in V(T)})$ be a tree-decomposition of a graph~$G$.
    Then, there exists a node~$t$ of~$T$ and a vertex~$v$ of~$G$ such that~${N[v]\subseteq B_t}$.
\end{lemma}

Using this, the following lemma shows the relation between tree-independence number and $\alpha$-degeneracy. 

\begin{restatable}{lemma}{lemtreealphadeg}
    \label{lem:tree-alpha-deg}
    For every graph~$G$, the $\alpha$-degeneracy of~$G$ is at most its tree-independence number. 
\end{restatable}

\begin{proof}
    Let~$G$ be a graph, let~$k$ be the tree-independence number of~$G$, and let ${\mathcal{T} = (T,\{B_t\}_{t \in V(T)})}$ be a tree-decomposition of~$G$ with independence number~$k$.
    Let~$H$ be a non-null induced subgraph of~$G$ and let~$\mathcal{T}_H$ be obtained from~$\mathcal{T}$ by restricting every bag to its intersection with~$V(H)$.
    It is straightforward to verify that~$\mathcal{T}_H$ is a tree-decomposition of~$H$ with independence number at most~$k$.
    By \zcref{lem:closed-neighborhoods}, there exists a node~$t$ of~$T$ and a vertex~$v$ of~$H$ such that ${N_H[v]\subseteq B_t\cap V(H)}$.
    Hence, ${\alpha(N_H[v]) \leq \alpha(H[B_t\cap V(H)]) \leq \alpha(\mathcal{T}_H) \leq k}$.
\end{proof}

\section{%
\texorpdfstring{Bounding $\alpha$-degeneracy}%
{Bounding alpha-degeneracy}}
\label{sec:degeneracy}

In this section we discuss $\alpha$-degeneracy and prove \zcref{thm:alpha-degeneracy-P5-Kll-simplified,thm:alpha-degeneracy-Sd-Kll-simplified,thm:dbs,prop:lower-bound}.

\subsection{%
\texorpdfstring{$P_5$-free and $S_d$-free graphs}%
{P5-free and Sd-free graphs}}
\label{sec:P5_Sd_degeneracy}
Recall that for $d \geq 2$, by $S_d$ we denote the graph obtained from the $d$-leaf star $K_{1,d}$ by subdividing each edge once. 
Note that $S_2 = P_5$.

We start with a lemma that follows from Lemma 11 of Bonamy, Bousquet, Pilipczuk, Rzążewski, Thomassé, and Walczak~\cite{BonamyBPRTW2022}. 
We include a proof for completeness. 

\begin{lemma}
    \label{lem:noKll}
    Let ${p}$ and ${\ell}$ be positive integers and let $H$ be a bipartite $K_{\ell,\ell}$-free graph with bipartition $A,B$, where $\abs{B} \geq p \cdot \ell$.
    Then at most $\ell-1$ vertices from $A$ have fewer than~$p$ non-neighbors in~$B$. 
\end{lemma}

\begin{proof}
    Suppose that there is a set $A' \subseteq A$ of size $\ell$, such that each $a \in A'$ has at most $p-1$ non-neighbors in $B$.
    This means that the number of common neighbors of all vertices of $A'$ in $B$ is at least $\ell \cdot p - \ell(p-1) = \ell$. This yields a $K_{\ell,\ell}$ in $H$, a contradiction.
\end{proof}

The next lemma follows the idea of a step in the proof of Theorem 3 in~\cite{BonamyBPRTW2022}.

\begin{restatable}{lemma}{lembigdegree}
    \label{lem:bigdegree}
    Let ${d,\ell \geq 2}$ be integers and let $H$ be a $\{dK_2,K_{\ell,\ell}\}$-free bipartite graph with bipartition~${X,Y}$.
    Suppose that for every $x \in X$ it holds that $\deg(x) \geq \ell^{d-1}$.
    Then $\abs{X} \leq \binom{d}{2}(\ell-1)$. 
\end{restatable}

\begin{proof}
    For each vertex~$v$ and set~${A \subseteq V(G)}$, let~$P(v,A)$ denote~$N(v) \setminus N(A \setminus \{v\})$, that is, the set of vertices that are adjacent to~$v$ but not adjacent to any~${a \in A \setminus \{v\}}$. 
    (Note that $v$ need not belong to $A$.)
    We call $P(v,A)$ the \emph{private neighborhood of~$v$ with respect to~$A$}.
    Suppose that $\abs{X} > \binom{d}{2}(\ell-1)$. 
    The goal is to construct an induced $dK_2$ in the graph, leading to a contradiction. 
    We therefore find $d$ vertices $z_1,\dots,z_{d-1},x \in X$ which have a private neighbor each. In the following, we define $Z_j=\{z_1,\ldots,z_{j}\}$ for~${j \in \{-1,0, \dots, d-1\}}$. 
    Note that $Z_0=Z_{-1}=\emptyset$. We will define $z_j$ recursively. 
    For that, we keep track of a set $X_{j-1}$ of candidates for $z_{j}$. 
    We set $X_{-1} \coloneqq  X_0 \coloneqq  X$ and recursively define $z_j$ and $X_j$ for $j \in \{0, 1, \dots, d-1\}$ such that 
    \begin{enumerate}
        [label=(\arabic*)]
        \item \label{it:sizeXj} $X_{j}$ is of size more than $\frac{ d(d-1)-j(j+1)}{2} (\ell-1)$ with~$X_{j} \subseteq X_{j-1} \setminus Z_{j-1}$, 
        \item \label{it:zjinXj} if~${j > 0}$, then $z_{j} \in X_{j-1}$, and
	\item \label{it:star} for every~${x \in X_j}$ and every~${z \in Z_j \cup \{x\}}$ it holds that 
	\[\abs{ P(z, Z_j \cup \{x\}) } \geq \ell^{d-1-j}. \]
    \end{enumerate}
    
    Note that proving this claim is enough to reach a contradiction. 
    Indeed, for ${j=d-1}$, by selecting an arbitrary vertex~$x \in X_{d-1}$ (which is a nonempty set by condition \ref{it:sizeXj}) and any~${z \in Z_{d-1} \cup \{x\}}$, it holds 
    that $\abs{P(z,Z_{d-1} \cup \{x\})} \geq \ell^0 = 1$. 
    Thus, choosing for each $z \in Z_{d-1} \cup \{x\}$ any vertex in $P(z,Z_{d-1} \cup \{x\})$, we obtain 
    an induced $dK_2$ in $H$, a contradiction.

    Note that for~$j = 0$, all the conditions are satisfied since for each~${x \in X_0 = X}$, we have~$\abs{P(x,\{x\})} = \abs{N(x)} \geq \ell^{d-1}$, and $X_{0}= X\subseteq X_{-1} \setminus Z_{-1}= X$.
    So suppose inductively that $1 \leq j \leq d-1$ and we have already selected sets $Z_{j-1}$ and $X_{j-1}$ so that these conditions are satisfied. 
    Since~${\frac{ d(d-1)-(j-1)j}{2} (\ell-1)\geq 0}$, the set $X_{j-1}$ is nonempty. 
    Let $z_j$ be a vertex in $X_{j-1}$ with the minimum number of neighbors in the set $Y\setminus N(Z_{j-1})$; in other words, 
    \[
        \abs{P(z_j,Z_{j-1})} \leq \abs{P(x,Z_{j-1})}
    \]
    for all~${x \in X_{j-1}}$.
    In particular, condition \ref{it:zjinXj} is satisfied for~$j$.

    We next define $X_j$. To obtain $X_j$ from $X_{j-1}$ we need to discard every vertex from $X_{j-1}$ which violates property $(3)$ for $Z_j$. We make this more precise in the following.
    Let~${p = \ell^{d-1-j}}$. For $z\in Z_j$, let $X^z$ be the set of vertices in $X_{j-1}$ with fewer than $p$ non-neighbors in $P(z,Z_j)$ (that is, for $x\in X_z$ the inequality in \ref{it:star} does not hold with $z$ and $Z_j$). We set ${X_j \coloneqq X_{j-1} \setminus \bigcup_{z \in Z_j} X^z}$. 
    
    We argue that conditions \ref{it:sizeXj} and \ref{it:star} hold for $X_j$.
    By property~\ref{it:star}, for~${z_j \in X_{j-1}}$ and ${z \in Z_{j-1} \cup \{z_j\} = Z_j}$, we have ${\abs{P(z,Z_j)} \geq \ell^{d-1-j+1} = p \cdot \ell}$. 
    By \zcref{lem:noKll}, we obtain that ${\abs{X^z} \leq \ell-1}$ for each $z\in Z_j$. Hence observe that $\abs{\bigcup_{z \in Z_j} X^z} \leq j (\ell-1)$. 
    We thus have 
    \begin{align*}
        \abs{X_j} 
        &\geq \abs{X_{j-1}} - j(\ell-1) > \frac{ d(d-1)-(j-1)j}{2} (\ell-1) - j \cdot (\ell-1)\\
        &= \frac{ d(d-1)-j(j+1)}{2} (\ell-1), 
    \end{align*}
    as required.
    Furthermore, since $z_j$ has no non-neighbors in $P(z_j,Z_j)$, it holds that $z_j \in X^{z_j}$, so $z_j \notin X_j$. 
    This implies that $X_j\subseteq X_{j-1} \setminus \{z_j\}$, so condition \ref{it:sizeXj} is satisfied for~$j$. 
    
    Now let us argue that condition \ref{it:star} holds for~$X_j$ and~$Z_j$. 
    Pick any~${x \in X_j}$ and first consider~${z \in Z_j}$. 
    The definition of $X_j$ implies that~${x \notin X^z}$; hence, $x$ has at least~${p = \ell^{d-1-j}}$ non-neighbors in $P(z,Z_j)$. 
    Thus, ${\abs{P(z,Z_j \cup \{x\})} \geq \ell^{d-1-j}}$, as claimed. 
    
    Now consider~${z = x}$. 
    Recall that ${\abs{P(z_j,Z_{j-1})} \leq \abs{P(x,Z_{j-1})}}$ by the choice of~$z_j$. 
    Thus, we obtain
    \begin{align*}
        \abs{P(x,Z_j\cup \{x\})} 
        &= \abs{P(x,Z_j)}\\
        &= \abs{P(x,Z_{j-1})} - \abs{(N(x) \cap N(z_j)) \setminus N(Z_{j-1})} \\
        &\geq \abs{P(z_j,Z_{j-1})} - \abs{(N(x) \cap N(z_j)) \setminus N(Z_{j-1})} \\
        &= \abs{P(z_j,Z_{j-1}\cup\{x\})} \\
        &= \abs{P(z_j,Z_{j}\cup\{x\})}\\ 
        &\geq \ell^{d-1-j},
    \end{align*}
    where the last inequality holds as argued in the previous paragraph (with~$z_j$ in place of~$z$). 
    This completes the proof of the claim and yields a contradiction.
    Thus, ${\abs{X} \leq \binom{d}{2}(\ell-1)}$.
\end{proof}

We will also need the following lemma. 

\begin{restatable}{lemma}{lemsmalldegree}
    \label{lem:smalldegree}
    Let $q,d$ be positive integers and let~$H$ be a bipartite graph with bipartition $X,Y$ that contains a matching~$M$ covering~$X$ and for each~${x \in X}$ it holds that ${\deg(x) \leq q}$.
    If ${\abs{X} > 2(d-1)q}$, then~$H$ contains an induced matching with~$d$ edges. 
\end{restatable}

\begin{proof}
	The proof is by induction on $d$.
	If $d=1$, then $\abs{X} > 0 $ and for an arbitrary vertex $x\in X$, the edge of $M$ having $x$ as an endpoint forms an induced matching.
	Suppose that $d \geq 2$ and the lemma holds for $d-1$.

	Let $Y'$ be the subset of $Y$ consisting of the vertices covered by $M$, and let $H' = H[X \cup Y']$.
	Clearly $\abs{Y'} = \abs{X}$ and $M$ is a perfect matching in $H'$.	
	For a vertex $v \in X \cup Y'$, let $M(v)$ be the vertex matched with $v$ by $M$.	
	As $\abs{E(H')} \leq q \cdot \abs{X} = q \cdot \abs{Y'}$, there is a $y \in Y'$ with $\deg(y) \leq q$.
	Let $x = M(y)$ and define
	\begin{align*}
		\widehat{Y} = & Y' \setminus N(x),\\
		\widehat{X} = & X \setminus N(y) \setminus \{ M(y') \mid y' \in N(x) \cap Y'\},\\
		\widehat{H} = & H'[\widehat{X} \cup \widehat{Y}].
	\end{align*}
	Notice that $\abs{\widehat{X}} \geq \abs{X} - 2q > 2(d-2)q$. Furthermore, there is a matching covering $\widehat{X}$ in $\widehat{H}$, 
	consisting of edges of $M$ with both endpoints in $\widehat{X} \cup \widehat{Y}$.
	Thus, by the induction hypothesis, $\widehat{H}$ has an induced matching with $d-1$ edges.
	Together with the edge $xy$ this gives an induced matching of $d$ edges in $H$.
\end{proof}

Now we show the following result, which in particular implies \zcref{thm:alpha-degeneracy-P5-Kll-simplified,thm:alpha-degeneracy-Sd-Kll-simplified}.

\begin{theorem}\label{thm:alpha-degeneracy-Sd-Kll}
    Let ${d,\ell \geq 2}$ be integers, let~$G$ be an $\{S_d, K_{\ell,\ell}\}$-free graph, and let~$I$ be a maximum independent set in~$G$.
    Then, ${\alpha(N[v]) < d^2 \ell + 2d\ell^{d-1}}$ for every ${v \in I}$. 
    Furthermore, for $d=2$, the bound is ${\alpha(N[v]) \leq 2 \ell - 2}$.
\end{theorem}

\begin{proof}
    Let $G$ and $I$ be as in the statement of the theorem, and consider any $v \in I$.
    Let $J$ be a largest independent set in $N(v)$. 
    Clearly $J \cap I = \emptyset$, as $v \in I$ and $v$ is adjacent to every vertex from $J$.
    
    Let $I' = I \setminus \{v\}$ and consider the graph $H$ induced by $J \cup I'$; note that it is bipartite with bipartition $J, I'$.
    Observe that $\alpha(H) \leq \abs{I'}+1$, as otherwise $I$ is not a maximum independent set in $G$.
    Thus, the size of a minimum vertex cover in $H$ is at least $\abs{J}-1$. Consequently, by K\"onig's theorem, $H$ has a matching $M$ of size at least $\abs{J}-1$.
    Let $J'$ be the set of vertices from $J$ covered by $M$. Clearly $\abs{J'} \geq \abs{J}-1$.
    
    Let $H'$ be the subgraph of $H$ induced by $J' \cup I'$. 
    It is bipartite and $M$ is a matching in $H'$ covering $J'$.
    Clearly $H'$ is $K_{\ell,\ell}$-free, as it is an induced subgraph of $G$.
    Furthermore, $H'$ is $dK_2$-free. Indeed, an induced matching with $d$ edges in $H'$, together with the vertex $v$, would induce a copy of $S_d$ in $G$.
    We aim to bound the size of $J$. 
    In general this is an easy application of Ramsey's theorem, but since we want to optimize the dependence on $\ell$, we need a more careful argument. 
    
    \paragraph{The general case}
    Let $J'_1$ be the set of vertices in $J'$ whose degree in $H'$ is at least $\ell^{d-1}$, and $J'_2 = J \setminus J'_1$.
    By \zcref{lem:bigdegree} applied to the graph $H'[J'_1 \cup I']$ we observe that $\abs{J'_1} \leq \binom{d}{2} (\ell-1) < d^2 \ell$.
    On the other hand, as each vertex of $J'_2$ has degree at most $\ell^{d-1}-1$ in $H'$, by \zcref{lem:smalldegree} applied to the graph $H'[J'_2 \cup I']$ we observe that $\abs{J'_2} \leq 2(d-1)(\ell^{d-1}-1) \leq 2d \ell^{d-1}$.
    Summing up, we obtain that $\abs{J'} < d^2 \ell + 2d \ell^{d-1}$ and thus $\abs{J} \leq d^2 \ell + 2d \ell^{d-1}$.
    
    \paragraph{Better bound for $d=2$}
    The above bound (if carefully inspected) only gives $\abs{J} < 3\ell-1$ for $d=2$.
    Here we improve it to ${\abs{J} < 2\ell-1}$.
    For a vertex $x \in J'$, let $M(x)$ be the vertex matched with $x$ by $M$.
    Towards contradiction suppose that ${\abs{J} \geq 2\ell-1}$ and thus ${\abs{J'} \geq 2\ell-2}$. 
    As $H'$ is $2K_2$-free, there is an ordering $x_1,x_2,\ldots,x_{\abs{J'}}$ of $J'$ such that
    \[N_{H'}(x_1) \subseteq N_{H'}(x_2) \subseteq \ldots \subseteq N_{H'}(x_{\abs{J'}})\,.\]
    In particular, all vertices $x_i$ for ${i \geq \ell-1}$ are adjacent to ${M(x_1),\ldots,M(x_{\ell-1})}$.
    Thus, vertices 
    \[
        \{x_{\ell-1},x_\ell,x_{\ell+1},\ldots,x_{2\ell-2}\} \cup \{v,M(x_1),\ldots,M(x_{\ell-1}) \}
    \]
    form an induced $K_{\ell,\ell}$ in $G$, a contradiction. 
\end{proof}

\subsection{A lower bound}

Let us provide the construction for \zcref{prop:lower-bound}, which we restate for the reader's convenience.

\proplb*

\begin{proof}
    Let~$G_\ell$ denote the graph obtained from the $5$-cycle by replacing each vertex with an independent set of size~$\ell-1$, where two vertices are adjacent if and only if the corresponding vertices in~$C_5$ are adjacent. 
    Since $G_\ell$ contains an induced $K_{\ell-1,\ell-1}$, the induced biclique number of $G_\ell$ is at least $\ell-1$.  
    It can be easily verified that $G_\ell$ is $\{2K_2,K_{\ell,\ell}\}$-free, and~$\alpha(G_\ell) = 2\ell-2$. 
    In particular, the induced biclique number of $G_\ell$ equals~$\ell-1$.
    Moreover, the neighbourhood of each vertex is an independent set of size~$2\ell-2$. 
    Since by \zcref{lem:tree-alpha-deg} the $\alpha$-degeneracy of $G_\ell$ is at most its tree-independence number, which is bounded by its independence number, the result follows. 
\end{proof}

\subsection{Classes with dominated balanced separators}

In what follows, $\log$ denotes the logarithm with base 2.

\thmdbs*

\begin{proof}
    First, for every positive $r \in \mathbb{R}$ we have $r \geq 1 +\ln r$. 
    By letting $r=\frac{d-1}{d}$ we obtain
    \[\frac{d-1}{d} \geq 1 +\ln \frac{d-1}{d}\,,\]
    thus 
    \[d-1 \geq d + d\ln \frac{d-1}{d}\,.\]
    Since for $r \leq 1$ we have $\log r \leq \ln r$, from there we observe that
    \begin{equation}
        d \log \frac{d-1}{d} +1 \leq d \ln \frac{d-1}{d} +1 \leq 0. \label{eq:estimatelogd}
    \end{equation}
    
    The proof of the theorem is by induction on $n$. 
    If $n=2$, then any vertex satisfies the claim.
    So let $G$ be a $K_{2,\ell}$-free graph from $\mathcal G$ with $n > 2$ vertices and suppose that the result holds for all such graphs with fewer vertices.
    
    If $G$ is a complete graph, then for every $v \in V(G)$ it holds that $\alpha(N[v])=1$ and we are done.  
    So we may assume that $G$ is not complete.
    Let $D$ be the set of universal vertices in $G$, that is, the set of vertices  adjacent to all other vertices.
    Note that $D$ is a clique in $G$ and every vertex in $D$ is adjacent to every vertex of $V(G)\setminus D$.
    Let $G'=G-D$.
    Since $G$ is not complete, $n' \coloneqq  \abs{V(G')} \geq  2$. 
    Note that $G' \in \mathcal{G}$ as $\mathcal{G}$ is hereditary.
    Furthermore, $G'$ is $K_{2,\ell}$-free and has no universal vertices.
    We consider two cases.

    \medskip

    \noindent\textbf{Case 1. $G'$ has a vertex $x$ of degree at least $\frac{1}{d} n'-1$.}
        Note that every component of $G' - N[x]$ has at most $\frac{d-1}{d}n'$ vertices.
        Furthermore, there is at least one such component $C$, as $x$ is not universal.

        We apply induction to $C$, obtaining a vertex $v$ such that
        \begin{align*}
        \alpha(N[v] \cap C) &\leq d\ell \log(\abs{C}) \leq d \ell \log\left (\frac{d-1}{d} n' \right) \\
                            &= d\ell \bigg( \log \frac{d-1}{d} + \log n'\bigg) =  d\ell \log \frac{d-1}{d} + d\ell \log n'.
        \end{align*}
        On the other hand, we observe that that $\alpha(N(v) \cap N(x)) \leq \ell-1$, as $G$ is $K_{2,\ell}$-free. Finally, $\alpha(N(v) \cap D) \leq 1$, as $D$ induces a clique.
        Summing up, we obtain
        \begin{align*}
        \alpha(N[v]) &\leq \alpha(N[v] \cap C) + \alpha(N(v) \cap N(x))+ \alpha(N(v) \cap D) \\
                       &\leq d\ell \log \frac{d-1}{d} + d\ell \log n' + \ell-1 +1 \\
                       &\leq d\ell \log n' + \ell\left( d \log \frac{d-1}{d} + 1 \right) \\
                       & \leq d\ell \log n' \leq d\ell \log n,
        \end{align*}
        as required, where the penultimate inequality follows from~\eqref{eq:estimatelogd}.

\smallskip
    \noindent\textbf{Case 2. $G'$ has maximum degree smaller than $\frac{1}{d} n' -1$.}
        Let $X \subseteq V(G')$ be a set of size at most $d$ such that every component of $G' - N[X]$ has at most $\frac{1}{2}n'$ vertices; such a set exists as $G' \in \mathcal{G}$ and $\mathcal{G}$ has $d$-dominated balanced separators.
        Furthermore, there is at least one component $C$ of $G' - N[X]$, as 
        ${\abs{N_{G'}[X]} \leq d\cdot (\frac{1}{d} n'-1)<n'}$.        
        Again, we apply induction to $C$, obtaining a vertex $v$ in $C$ such that
        \begin{align*}
        \alpha(N[v] \cap C) \leq   d\ell \log(\abs{C}) \leq d \ell \log(n'/2) =  d\ell \log n' - d\ell.
        \end{align*}
          On the other hand, we observe that $\alpha(N(v) \cap N(X)) \leq d(\ell-1)$.
          Indeed, suppose that there is an independent set $I \subseteq N(v) \cap N(X)$ of size at least $d(\ell-1)+1$.
          As $\abs{X} \leq d$, by the pigeonhole principle, there is a vertex $x \in X$ such that $\abs{I \cap N(x)} \geq \ell$, thus $G$ contains an induced $K_{2,\ell}$, a contradiction.
          Summing up, we obtain
         \begin{align*}
         \alpha(N[v]) &\leq \alpha(N[v] \cap C) + \alpha(N(v) \cap N(X)) + \alpha(N(v) \cap D) \\
                        &\leq d\ell \log n' - d\ell + d(\ell-1) +1 \\
                        &\leq d\ell \log n' \leq d\ell \log n,
         \end{align*}
        as required. This completes the proof.
\end{proof}

\section{%
\texorpdfstring{Bounded tree-independence number for $\{P_5, K_{\ell,\ell}\}$-free graphs}%
{Bounded tree-independence number for {P5, Kll}-free graphs}}

In this section, we prove \zcref{thm:P5treealpha}, which we restate here for convenience. 

\thmtreealpha*

The proof of the theorem is based on the following technical lemma.

\begin{restatable}{lemma}{bagwithNr}\label{lem:bagwithNr}
    Let $\ell\ge 2$ be an integer, let~$G$ be a $\{P_5, K_{\ell,\ell}\}$-free graph with at least two vertices, let $r \in V(G)$ be a vertex with ${\alpha(N[r]) \leq 2\ell-2}$, and let ${\mathcal{T} = (T,\{B_t\}_{t\in V(T)})}$ be a tree-decomposition of~${G - r}$
    such that 
    \begin{enumerate}
        [label=(P\arabic*)]
        \item\label{it:induction_independence} The independence number of each bag is at most ${4\ell-4}$.
        \item\label{it:induction_maximal_dist} Among tree-decompositions satisfying \ref{it:induction_independence}, the number of vertex pairs $v,w \in N(r)$ with $v\neq w$ for which there exists $t \in V(T)$ such that $v,w\in B_t$, is maximized.
    \end{enumerate}
    Then, each pair of vertices in $N(r)$ belongs to a common bag.
\end{restatable}

Let us argue that \zcref{lem:bagwithNr} indeed implies \zcref{thm:P5treealpha}.

\begin{proof}[Proof of \zcref{thm:P5treealpha} using \zcref{lem:bagwithNr}]
    The proof is by induction on $\abs{V(G)}$; the base case of $\abs{V(G)} = 1$ is trivial.
    For ${\abs{V(G)} \geq 2}$, we pick a vertex ${r \in V(G)}$ with ${\alpha(N[r]) \leq 2\ell-2}$ using \zcref{thm:alpha-degeneracy-P5-Kll-simplified}.
    By the induction hypothesis, there is a tree-decomposition of $G - r$ such that the independence number of each bag is at most~${4\ell-4}$.
    In particular, we pick a tree-decomposition $\mathcal{T} = (T,\{B_t\}_{t\in V(T)})$ of $G - r$ that satisfies \ref{it:induction_independence} and \ref{it:induction_maximal_dist}.
    By \zcref{lem:bagwithNr}, each pair of vertices in $N(r)$ belongs to a common bag.
    By the Helly properties of trees, if for every pair of vertices $v,w\in N(r)$ the subtrees $T(v)$ and $T(w)$ have a common node, there is a node $t \in V(T)$ which is contained in $T(v)$ for every $v\in N(r)$.
    That is, the bag $B_t$ contains all vertices of $N(r)$.
    Then we construct a tree-decomposition of $G$ by adding to $T$ a new leaf node $t'$ adjacent to $t$ and set its bag $B_{t'} = N[r]$.
    This is clearly a tree-decomposition of $G$ witnessing that ${\tin(G) \leq 4\ell-4}$.
\end{proof}

\subsection*{%
\texorpdfstring{Proof of \zcref{lem:bagwithNr}}%
{Proof of Lemma 4.1}}
% In the remainder of the section, we prove \zcref{lem:bagwithNr}. 

\ 

\paragraph{Setting}
Let~$G' \coloneqq G - r$ and let~$\mathcal{T}$ be a tree-decomposition as in the premise of \zcref{lem:bagwithNr}. 
For a vertex $v \in N(r)$, we define $\NN(v) \coloneqq N(v) \setminus N[r]$.

Let $\mathcal{P}$ denote the set of ordered pairs $(x,y)$, where $x,y$ are distinct neighbors of $r$ that are not in a common bag of $\mathcal{T}$. 
Note that~$x$ and~$y$ are nonadjacent for all~${(x,y) \in \mathcal{P}}$. 
To prove \zcref{lem:bagwithNr}, we need to show that~$\mathcal{P} = \emptyset$. 

For a contradiction, suppose that~$\mathcal{P}$ is nonempty. 
We say that a pair $(x,y) \in \mathcal{P}$ is \emph{bad} if $\alpha(\NN(x)\setminus \NN(y)) \geq \ell$.
We fix a specific pair $(x,y) \in \mathcal{P}$ as follows:
If possible, we pick $(x,y)$ to be a bad pair, and subject to this rule, we pick $(x,y)$ that maximizes the distance between $T(x)$ and $T(y)$ in~$T$.

\paragraph{Neighbors of $r,x,$ and $y$}

We denote $N(r) \cup \NN(x) \cup \NN(y)$ by $M$; see also \zcref{fig:splitMbody}.
Let $U$ be the set of vertices $u\in N(r)$ that have a neighbor outside of $M \cup \{r\}$, that is, \[U = \{ u \in N(r) \mid \NN(u) \setminus (\NN(x) \cup \NN(y)) \neq \emptyset \}.\]
Note that neither $x$ nor $y$ belongs to $U$.
We further differentiate the vertices of $U$ based on the connections between $u \in U$ and $x,y$, as follows:
\begin{align*}
U_0 &\coloneqq \{u \in U \mid u \notin N(x) \cup N(y)\}\,, &
U_{x} &\coloneqq \{u \in U \mid u \in N(x) \setminus N(y)\}\,, \\
U_{y} &\coloneqq \{u \in U \mid u \in N(y) \setminus N(x)\}\,, &
U_{xy} &\coloneqq \{u \in U \mid u \in N(x) \cap N(y)\}\,.
\end{align*}

For the sake of brevity, let us also denote by~$W_{x}$, $W_{y}$ and~$W_{x,y}$ the sets~{${\NN(x) \setminus \NN(y)}$, ${\NN(y) \setminus \NN(x)}$} and~${\NN(x) \cap \NN(y)}$, respectively. 

\begin{figure}[htbp]
    \centering
    \includegraphics[scale=1.1]{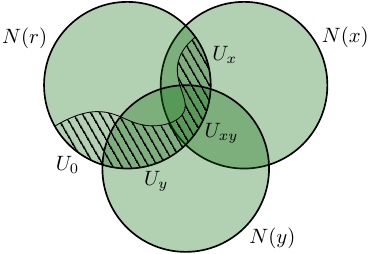}
    \caption{The green area denotes the set $M$. The striped area depicts the set $U$, i.e., vertices from $N(r)$ that have neighbors outside $M$.}
    \label{fig:splitMbody}
\end{figure}

\begin{observation}
    \label{claimcollection1}
    Let $C$ be a component of $G' - M$.
Then, the following holds.
    \begin{enumerate}
        [label=(\alph*)]
        \item\label{claim1-1} $N(C) \subseteq U \cup W_{xy}$. 
        \item\label{claim1-2} $C$ is complete to $N(C)\cap (U_0 \cup U_x\cup U_y)$. 
    \end{enumerate}
\end{observation}

\begin{proof}
    For \ref{claim1-1}, by the definition of~$U$, each vertex outside of~$M$ has all its neighbors in~$\NN(x) \cup \NN(y) = W_x \cup W_y \cup W_{xy}$. 
    Suppose for a contradiction that $c \in C$ has a neighbor $w$ in, say, the set $W_x$ (the other case is symmetric).
    Then $G$ contains a copy of $P_5$ on vertices $c,w,x,r,y$; see \zcref{fig:C_adjacent_to_U_and_N2(x)capN2(y)-main}.

    \begin{figure}[htbp]
    \begin{minipage}[t]{.5\linewidth}
    \centering
    \includegraphics[page=4,scale=1.2]{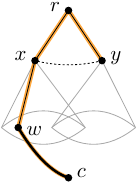}
    \subcaption{A copy of $P_5$ in \zcref{claimcollection1}\ref{claim1-1}.}
    \label{fig:C_adjacent_to_U_and_N2(x)capN2(y)-main}
    \end{minipage}
    \begin{minipage}[t]{.4\linewidth}
    \centering
    \includegraphics[page=5,scale=1.2]{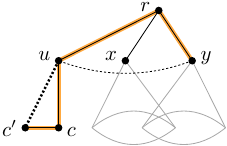}
    \subcaption{A copy of $P_5$ in \zcref{claimcollection1}\ref{claim1-2}.}
    \label{fig:C_complete_to_neighborhood_in_Uxy0-main}
    \end{minipage}
    \caption{The copies of $P_5$ from the proof of \zcref{claimcollection1}.}
    %\label{fig:2}}
    \end{figure}
    
    For \ref{claim1-2}, towards a contradiction, assume that the component $C$ is not complete to $N(C)\cap (U_0 \cup U_x\cup U_y)$, that is, there are two vertices $c,c' \in C$ and a vertex $u \in N(C)\cap (U_0 \cup U_x\cup U_y)$ such that $u$ is adjacent to $c$, but nonadjacent to $c'$.
    
    Moreover, we may assume that $c$ is adjacent to $c'$.
    Indeed, as $C$ is a component, there exists a path $P$ from $c$ to $c'$ within $C$.
    The path starts from a vertex adjacent to $u$ and ends in a vertex that is not adjacent to $u$.
    Hence, there are two consecutive vertices on $P$ such that one is adjacent to $u$ and the other is not.

    We claim that $G$ contains a $P_5$, leading to a contradiction.
    By symmetry, assume that $u \in U_0 \cup U_x$.
    Then $u$ is not adjacent to $y$, and $c',c,u,r,y$ is an induced $P_5$; see~\zcref{fig:C_complete_to_neighborhood_in_Uxy0-main}. 
\end{proof}

Towards proving \zcref{lem:bagwithNr}, we distinguish two cases, depending on whether the pair $(x,y)$ is bad or not.
The general outline of the proof in these two cases is similar, but the case that $(x,y)$ is not a bad pair is significantly less involved. 

\paragraph{Case 1: The pair $(x,y)$ is not bad}

By the choice of $(x,y)$, we know that there are no bad pairs in~$\mathcal{P}$ at all.
We let~$t_x$ and~$t_y$ be the closest nodes of~$T$ such that ${t_x \in T(x)}$ and~${t_y \in T(y)}$.
Note that~$t_x$ and~$t_y$ are distinct, since ${(x,y) \in \mathcal{P}}$, and uniquely defined since~$T$ is a tree.

We create a new tree-decomposition ${\mathcal{T}'=(T', \{B'_t\}_{t \in V(T')})}$ of~$G'$ satisfying~\ref{it:induction_independence} such that all the pairs of vertices $v,w \in N(r)$ that are in a common bag in $\mathcal{T}$ are in a common bag in $\mathcal{T}'$ as well, and additionally $x$ and $y$ appear in a common bag in $\mathcal{T}'$.
This gives a contradiction with property~\ref{it:induction_maximal_dist} of $\mathcal{T}$.

Let us start by defining $T'$; see \zcref{fig:no_bad_pair_decomposition-main}.
We create a ``master'' copy~$T^M$ of~$T$ and another copy~$T^C$ of~$T$ for every component~$C$ of~${G'-M}$.
For each ${t \in V(T)}$, we denote by~$t^M$ and~$t^C$ the copies of~$t$ in the respective copies~$T^M$ and~$T^C$ of~$T$.
Now, we create~$T'$ by taking the disjoint union of all these copies and adding edges ${t_y^Mt_y^C}$ for every component~$C$ of~${G'-M}$.

    \begin{figure}[htbp]
        \centering
        \includegraphics[width=0.9\textwidth]{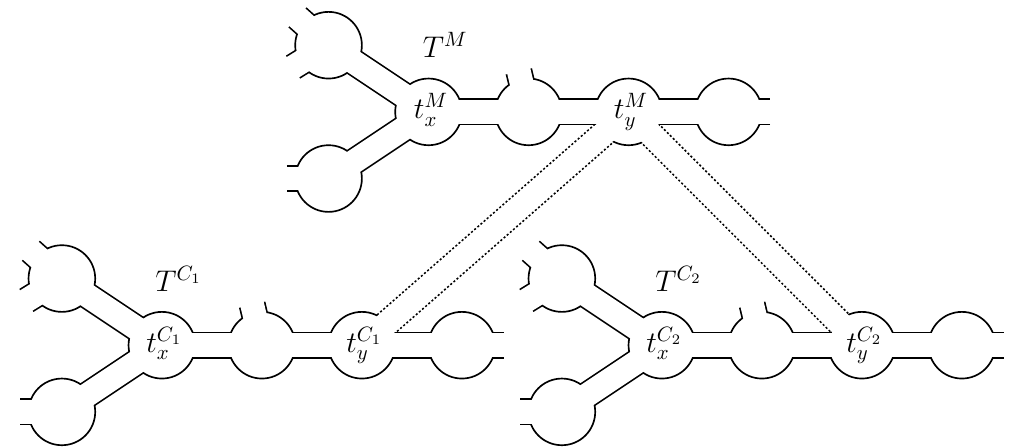}
        \caption{Tree $T'$ obtained by connecting $T^M$ and $T^C$ for each component $C$ (here we sketched just two components $C_1$ and $C_2$).}
        \label{fig:no_bad_pair_decomposition-main}
    \end{figure}
    
Now, let us define the bags $B'_t$ for $t \in V(T')$.
For each $t^M \in V(T^M)$, we do the following:
\begin{enumerate}
    \item We initialize $B'_{t^M} \coloneqq B_t \cap M$.
    \item For each vertex $u\in U$, we extend the subtree $T^M(u)$ minimally so it contains vertices $t^M_x$ and $t^M_y$. 
    In other words, we add the vertex $u$ to the bags of~$t^M_x$ and~$t^M_y$, and to the minimum number of other bags so that $T^M(u)$ is connected.
    \item We add the vertex $x$ to the bag of every node on the path from $t^M_x$ to $t^M_y$ in $T^M$.
\end{enumerate}
Now, consider a component $C$ of $G'-M$. 
For each $t^C\in V(T^C)$, we set 
\[
    B'_{t^C} \coloneqq \bigl( B_t \cap N[C] \bigr) \cup \bigl( N(C) \cap (U \setminus U_{xy}) \bigr).
\]
This completes the definition of $\mathcal{T}'$.
Note that for every $t^M \in V(T^M)$ we have $B'_{t^M} \subseteq M$.  

\begin{claim}
    $\mathcal{T}'$ is a tree-decomposition of~$G'$. 
\end{claim}

\begin{subproof}
    First, let us argue that for each vertex $v\in V(G')$, the set $T'(v)$ is nonempty. 
    Consider the case $v \in M$; the other one is similar.
    Note that $T(v)$ is nonempty by the definition of $\mathcal{T}$, and the copies of nodes in $T(v)$ in the tree $T^M$ are contained in $T'(v)$, as we were only adding vertices to the bags.
    This in particular means that $T'(v)$ is nonempty.

    It is clear from the construction that for each $v$ the set of bags where $v$ appears is connected when restricted to either $T^M$ or some $T^C$.
    
    Observe that the vertices that do not belong solely to the bags of nodes in $T^M$ or in $T^C$ for some $C$ are exactly those vertices $v$ that belong to the neighborhood of $C$ for some component~$C$ of $G' - M$.
    Thus, it remains to prove that for every such component $C$ and every $v \in N(C)$, the set $T'(v)$ is connected.
    For this, it is enough to show that $v$ belongs to the bags of both $t^M_y$ and $t^C_y$, since these two nodes connect $T^M$ and $T^C$.

    By \zcref{claimcollection1}\ref{claim1-1}, we have 
    $v \in U \cup W_{xy}$. 
    
    If $v$ is a common neighbor of $x$ and $y$, so in particular if 
    $v \in U_{xy} \cup W_{xy}$,
    we claim that~$v$ belongs to $B_{t_y}$ in $T$.
    Indeed, since $v$ is adjacent to both $x$ and $y$, the set $T(v)$ must intersect both $T(x)$ and $T(y)$ so, as $T$ is a tree, it must contain both $t_x$ and $t_y$.
    Thus, $v$ belongs to the bags of both $t^M_y$ and $t^C_y$.

    So, assume that $v \in U \setminus U_{xy}$.
    On one hand, $v$ belongs to the bag of $t^M_y$, as it was explicitly added there in the second step of the definition of $B'_{t^M}$ for $t^M \in V(T^M)$.
    On the other hand, $v$ belongs to the bags of all nodes in $T^C$, so in particular to the bag of~$t^C_y$.

    Lastly, we need to argue that for each edge $uv \in E(G')$, there is a node $t \in V(T')$ such that $u,v \in B'_t$. 
    Since $T$ is a tree-decomposition of $G'$, there is a node $t \in V(T)$ such that $u,v \in B_t$.
    If $u,v \in M$, then $u,v \in B'_{t^M}$.
    Otherwise, we may assume that $u$ belongs to some component $C$ of $G'-M$ and $v \in N[C]$.
    In this case we have $u,v \in B'_{t^C}$.
\end{subproof}

\begin{claim}
    It holds that ${\alpha(\mathcal{T}') \leq 4\ell-4}$.
\end{claim}

\begin{subproof}
    First, consider a node $t^M$ of $T^M$.
    Let $I$ be an independent set of $G'$ contained in $B'_{t^M}$.
    If $I \subseteq B_t$, we have ${\abs{I} \leq 4\ell-4}$ by the definition of $\mathcal{T}$.
    Hence, assume that $I$ contains some vertex that was added either in the second or the third step of the construction of $B'_{t^M}$.
    Such a vertex is either $x$ or some $u \in U$.
    Suppose first that $x \in I$. 
    As $I \subseteq B'_{t^M} \subseteq M = N(r) \cup \NN(x) \cup \NN(y)$ and $I\cap \NN(x) = \emptyset$, since $x\in I$, we have
    \[
        \abs{I} = \abs{I \cap N(r)} + \abs{I \cap W_y} \leq (2\ell-2) + (\ell - 1) = 3\ell - 3,
    \]
    by the choice of $r$ and since~$(y,x)$ is not a bad pair (as we are in the case that no pairs in~$\mathcal{P}$ are bad). 
    
    So suppose that there is a vertex $u \in I$ such that $u \in U$.
    Note that we can assume that $T(u) \cap T(x) = \emptyset$ or $T(u) \cap T(y) = \emptyset$, as otherwise $u$ must be present in both $B_{t_x}$ and $B_{t_y}$ and thus was not added in the second step of the construction of $B'_{t^M}$.
    Consider the case $T(u) \cap T(x) = \emptyset$; the other one is similar. 
    Since $u\in I$, we have $I\cap \NN(u) = \emptyset$.
    Furthermore, since $T(u) \cap T(x) = \emptyset$, 
    no bag of~$\mathcal{T}$ contains both~$u$ and~$x$, so~$(x,u) \in \mathcal{P}$. 
    This implies
    \begin{align*}
        \abs{I} &\leq \abs{I \cap N(r)} + \abs{I \cap \big(\NN(x) \setminus \NN(u)\big)} + \abs{I \cap W_y} \\
        &\leq (2\ell-2) + (\ell - 1) + (\ell - 1) = 4\ell-4,
    \end{align*}
    where $\abs{I \cap \big(\NN(x) \setminus \NN(u)\big)} \leq \ell - 1$ follows from the fact that $(x,u)$ is not a bad pair.
    
    Finally, consider a component~$C$ of~${G'-M}$ and a node $t^C$ in~$T^C$.
    Let ${I \subseteq B'_{t^C}}$ be an independent set.
    Again, if $I \subseteq B_t$, we have ${\abs{I} \leq 4\ell-4}$ by the assumption on~$\mathcal{T}$.
    Hence, assume that $I$ contains some of the new vertices from $N(C) \cap (U \setminus U_{xy})$.
    However, the component $C$ is complete to these vertices by \zcref{claimcollection1}\ref{claim1-2}.
    Thus, $I \subseteq N(C) \cap (U \setminus U_{xy}) \subseteq N(r)$, so ${\abs{I} \leq 2\ell-2}$ by the choice of~$r$.
\end{subproof}

To complete the proof of \zcref{lem:bagwithNr} for the case that $(x,y)$ is not a bad pair, let us argue that the number of pairs of vertices in $N(r)$ that appear in a single bag increased, compared to $\mathcal{T}$.
Consider a pair $x',y' \in N(r)$ of distinct vertices that appear in the bag of some node of $\mathcal{T}$, say $x',y' \in B_t$.
Note that $x',y' \in M$ and thus, $x',y' \in B'_{t^M}$.
Additionally, due to the third step of the construction of $B'_{t^M}$, we have $x,y \in B'_{t^M_y}$.
Thus, by \ref{it:induction_maximal_dist}, this contradicts the choice of $\mathcal{T}$ and completes the proof of \zcref{lem:bagwithNr} in this case.

\paragraph{Case 2: The pair $(x,y)$ is bad}

Let us briefly discuss the general strategy of how to proceed when~$(x,y)$ is a bad pair. 

In order to get a contradiction to the existence of such $x,y$ and hence to prove \zcref{lem:bagwithNr} in Case 2, we construct a tree-decomposition $\mathcal{T}'=(T',\{B'_t\}_{t \in V(T)})$ that satisfies \ref{it:induction_independence} and is superior to $\mathcal{T}$ with respect to \ref{it:induction_maximal_dist}.
As before, $T'$  is composed of a ``master'' subtree $T^M$ to which we join subtrees $T^D$ reminiscent of $T^C$ from Case~1.
However, the construction of the bags is more involved compared to Case~1, as we do not have a bound on  $\alpha(W_x)$. 
Thus, we cannot move vertices of $U$ as freely within the bags of nodes of $T^M$.

To be more specific, we construct the bags of nodes in $T^M$ in two steps.
First, we define intermediate bags $B''_{t^M}$ by intersecting the bag $B_t$ with $M$, adding some of the vertices from $U$ to the bag.
We also add the vertex $x$ to some bags, in order to ensure that the vertices $x$ and $y$ appear in a common bag.
This yields an intermediate decomposition $\mathcal{T}^M$ of $G'[M]$, which is similar to the final tree-decomposition from Case~1, restricted to $T^M$.

Unlike in the previous case, however, there might be vertices $c$ from some component $C$ of $G'-M$ whose neighborhood within $M$ is \emph{not} contained in a single bag of $\mathcal{T}^M$, which makes it difficult to cover the edges between $c$ and $N(c)$ in $T^C$ (as in the previous case).
Hence, as the second step, we define bags $B'_{t^M}$ by extending the bags  $B''_{t^M}$ by such vertices $c$,
so that we cover the problematic edges in the bags of nodes in $T^M$ instead.
Let us denote by $\cc$ the set of components of $G'-M$ and define $M' \coloneqq M \cup \bigcup_{C \in \cc} C'$, where $C'$ are the vertices from $C$ added to the bags of $T^M$. 
We show that ${\mathcal{T}}^{M'} \coloneqq (T^M,\{B'_t\}_{t \in V(T^M)})$ is actually a tree-decomposition of $G'[M']$.

The role of the nodes of $T^D$ is very similar to the role of $T^C$ from Case~1.
For each component $D$ of $G' - M'$, we create a copy $T^D$ of $T$ whose bags will be representing the adjacencies within $D$ and between $D$ and $N(D)$, yielding a tree-decomposition ${\mathcal{T}}^D$ of the graph with vertex set~$N[D]$ formed by edges that have at least one endpoint in~$D$. 
Moreover, the choice of nodes by which we join $T^D$ to $T^M$ is slightly more elaborate compared to Case~1, and possibly different for different components~$D$.

We present the proof as follows. 
After making some more observations about the neighbors of~$r$, $x$, and~$y$, we first construct the tree-decompositions ${\mathcal{T}}^M$, ${\mathcal{T}}^{M'}$, and ${\mathcal{T}}^D$, and prove some of their properties. 
Then we put everything together to derive the final contradiction.

\paragraph{More on the neighbors of~$r$, $x$, and~$y$}

Let us continue by making some observations related to the neighborhood structure in~$G'[M]$. 
Note that none of these need the fact that~$(x,y)$ is bad, but if we assume that~$(x,y)$ is bad, then \ref{claim2-2} implies that~$(y,x)$ is not bad. 

\begin{observation}
    \label{claimcollection2}
    The following statements hold. 
    \begin{enumerate}
    [label=(\alph*)]
        \item\label{claim2-1} $W_x$ is complete to~$W_y$. 
        \item\label{claim2-2} Either $\alpha(W_x) \leq \ell-1$, or $\alpha(W_y) \leq \ell-1$. 
        \item\label{claim2-3} $U_x$ is complete to $U_y$. 
        \item\label{claim2-4} $U_0$ is complete to $W_x \cup W_y$. 
        \item\label{claim2-5} $U_x$ is complete to $W_y$. 
        \item\label{claim2-6} $U_y$ is complete to $W_x$. 
        \item\label{claim2-7} Each component $C$ of $G' - M$  is complete to $N(C) \cap W_{xy}$. 
    \end{enumerate}
\end{observation}

\begin{proof}
   For~\ref{claim2-1}, note that if there exists a vertex $w_x \in W_x$ nonadjacent to some vertex $w_y \in W_y$, then the set ${\{w_x,x,r,y,w_y\}}$ induces a copy of~$P_5$; see \zcref{fig:diffs_are_complete}. 

    \begin{figure}[htbp]
        \begin{minipage}[t]{.32\linewidth}
            \centering
            \includegraphics[page=1,scale=1.2]{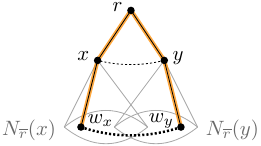}
            \subcaption{In the proof of~\ref{claim2-1}.}
            \label{fig:diffs_are_complete}
        \end{minipage}
        \begin{minipage}[t]{.32\linewidth}
            \centering
            \includegraphics[page=7,scale=1.2]{figures/find_p5.pdf}
            \subcaption{In the proof of~\ref{claim2-5}.}
            \label{fig:P5easy}
        \end{minipage}
        \begin{minipage}[t]{.32\linewidth}
            \centering
            \includegraphics[page=6,scale=1.1]{figures/find_p5.pdf}
            \subcaption{In the proof of~\ref{claim2-7}.}
            \label{fig:C_complete_to_neighborhood_in_N2(x)capN2(y)}
        \end{minipage}
            
        \begin{minipage}[t]{.32\linewidth}
        \centering
                \includegraphics[page=2,scale=1]{figures/find_p5.pdf}
                \subcaption{In the proof of~\ref{claim2-3} when $u_x$ and $u_y$ have a common neighbor.}
                \label{fig:U_x_and_U_y_complete_same}
        \end{minipage}
        \begin{minipage}[t]{.32\linewidth}
        \centering
                \includegraphics[page=19,scale=1]{figures/find_p5.pdf}
                 \subcaption{In the proof of~\ref{claim2-3} when $u_x$ and $u_y$ have no common neighbor and $w_xw_y \notin E(G)$.}
                \label{fig:U_x_and_U_y_complete_different-nonedge}
        \end{minipage}
        \begin{minipage}[t]{.32\linewidth}
        \centering
                \includegraphics[page=3,scale=1]{figures/find_p5.pdf}
                 \subcaption{In the proof of~\ref{claim2-3} when $u_x$ and $u_y$ have no common neighbor and $w_xw_y \in E(G)$.}
                \label{fig:U_x_and_U_y_complete_different-edge}
        \end{minipage}
        \caption{The various copies of $P_5$ in the proof of \zcref{claimcollection2}. Full lines show edges, dotted lines show non-edges; thickness signifies an assumption that led to a contradiction.}
        \label{fig:U_x_and_U_y_complete}
    \end{figure}
    For~\ref{claim2-2}, if both $\alpha(W_x) \geq \ell$ and $\alpha(W_y) \geq \ell$, then $G$ contains a copy of $K_{\ell,\ell}$, as the sets are complete to each other by \ref{claim2-1}. 

    To prove~\ref{claim2-3}, suppose for a contradiction that there are $u_x \in U_x$ and $u_y \in U_y$ that are nonadjacent. 
    By definition of~$U$, each of $u_x$ and $u_y$ has a neighbor outside of~$M$. 
    We distinguish two cases of how we find a copy of $P_5$. 
    If $u_x$ and $u_y$ have a common neighbor outside $M$, say $w_u$, we have a copy of $P_5$ with consecutive vertices $x,u_x,w_u,u_y,y$; see \zcref{fig:U_x_and_U_y_complete_same}. 
    So suppose that no such common neighbor exists.
    We choose $w_x \in \NN(u_x) \setminus (N(x) \cup N(y))$ and $w_y \in \NN(u_y) \setminus (N(x) \cup N(y))$ arbitrarily.
    By the failure of the previous case, we have $w_xu_y \notin E$ and $w_yu_x \notin E$.
    Moreover, $w_x$ must be adjacent to $ w_y$ as otherwise $\{w_x,u_x,r,u_y,w_y\}$ induces a copy of~$P_5$ in~$G$; see \zcref{fig:U_x_and_U_y_complete_different-nonedge}.
    However, this yields that~${G[\{x,u_x,w_x,w_y,u_y,y\}]}$ is isomorphic to~$P_6$; see \zcref{fig:U_x_and_U_y_complete_different-edge}. 

    For \ref{claim2-4}, \ref{claim2-5}, and \ref{claim2-6}, 
    note that it is enough to show that for every~${u \in U}$ such that $ux \notin E(G)$ or $uy \notin E(G)$, it holds that $u$ is adjacent to every vertex in~$W_x$ or $W_y$, respectively. 
    Fix an arbitrary $w_u \in N(u)\setminus M$; this vertex exists since $u \in U$. 
    Fix any $w_x \in W_x$; note that $w_x \notin U$ as $U \subseteq N(r)$. 
    Note that $w_u$ is not adjacent to $w_x$ by \zcref{claimcollection1}\ref{claim1-1}. 
    But then there is a $P_5$ on $w_x,x,r,u,w_u$; see \zcref{fig:P5easy}, unless $u$ is adjacent to $w_x$.
    The other case with $uy \notin E(G)$ is symmetric. 

    Lastly, to prove \ref{claim2-7}, suppose for a contradiction that $C$ is not complete to $N(C) \cap \NN(x)\cap \NN(y)$. 
    Then there are adjacent vertices $c, c' \in C$ and $w \in \NN(x)\cap \NN(y)$ such that $w$ is adjacent to $c$ but not to $c'$ (the adjacency of~$c$ and~$c'$ can be assumed as in the proof of \zcref{claimcollection1}\ref{claim1-2}). 
    Then $G$ contains a copy of $P_5$ on vertices $c',c,w,x,r$; see \zcref{fig:C_complete_to_neighborhood_in_N2(x)capN2(y)}.
\end{proof}

\paragraph{Defining the tree-decomposition $\mathcal{T}^M$}
Let $T^M$ be a copy of $T$. 
Again, we will denote the copy of a node $t \in V(T)$ in $T^M$ by $t^M$.

We define the bags $\{B''_{t^M}\}_{t^M \in V(T^M)}$ as follows.
\begin{itemize}
    \item For each $t \in V(T)$, we initialize $B''_{t^M} \coloneqq B_{t} \cap M$.
    \item Let us call a vertex~${u \in U}$ \emph{movable} if ${\alpha\left(\NN(x)\setminus \NN(u)\right) \leq \ell-1}$ and \linebreak ${\NN(x)\setminus \NN(u)\subseteq \NN(y)}$.
    Let~$U_{\mathrm m}$ be the set of all movable vertices. 
    \item For each~${t \in V(T)}$, we add each movable vertex to the bag~$B''_{t^M}$.
    \item For all nodes $t$ in $T$ that are on the path from $t_x$ to $t_y$ in~$T$, we add $x$ and all vertices from $U_y \setminus U_{\mathrm m}$ to $B''_{t^M}$. 
\end{itemize}

Our goal is to show that $\mathcal{T}^M \coloneqq (T^M, \{B''_t\}_{t \in V(T^M)})$ is a tree-decomposition of $G'[M]$ with ${\alpha(\mathcal{T}^M) \leq 4\ell-4}$.
To do so, we start with observing some properties of the movable vertices $u\in U_{\mathrm m}$.

\begin{claim}\label{cl:U0new}
    $U_0\subseteq U_{\mathrm m}$. 
\end{claim}

\begin{subproof}
    By \zcref{claimcollection2}\ref{claim2-4}, it suffices to show that for each $u_0 \in U_0$ the set~$S$ of non-neighbors of $u_0$ in $\NN(x)\cap \NN(y)$ satisfies $\alpha(S) \leq \ell-1$.
    Note that for each $w_x \in W_x$ and $w_s \in S$, we have $w_x$ adjacent to $w_s$, as otherwise there is a copy of $P_5$ on vertices $u_0, w_x, x, w_s, y$; see \zcref{fig:anothereasyp5}.
    Then, since $\alpha(W_x) \geq \ell$, we have $\alpha(S)\leq \ell-1$, as otherwise there is a copy of~$K_{\ell,\ell}$.
\end{subproof}

\begin{figure}[htbp]
    \centering
    \includegraphics[page=8,scale=1.2]{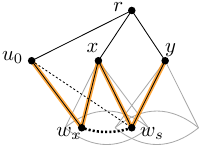}
    \caption{A copy of $P_5$ in \zcref{cl:U0new}.}
    \label{fig:anothereasyp5}
\end{figure}
   
\begin{claim}\label{claim:role_of_Uy}
$U_y\subseteq U_{\mathrm m}\cup B_{t_y}$.    
\end{claim}
\begin{subproof}
    Let $u_y\in U_y$.
    Suppose that $u_y \notin B_{t_y}$.
    Since $u_y$ is adjacent to $y$, the sets $T(u_y)$ and $T(y)$ have a nonempty intersection.
    As $t_y$ is the closest node to $t_x$ in $T$ whose bag contains~$y$, and $u_y \notin B_{t_y}$, it follows that the distance in $T$ between $T(x)$ and $T(y)$ is smaller than the distance in $T$ between $T(x)$ and $T(u_y)$.
    We chose the bad pair $(x,y)$ to maximize the distance between $T(x)$ and $T(y)$, so $(x,u_y)$ cannot be a bad pair.
    Hence, by definition, we have $\alpha(\NN(x) \setminus \NN(u_y)) \leq \ell-1$.
    Additionally, we have $W_x \subseteq \NN(u_y)$ by~\zcref{claimcollection2}\ref{claim2-6}.
    Thus, $u_y \in U_{\mathrm m}$.
\end{subproof}

    We observe that $\mathcal{T}^M$ is a tree-decomposition of $G'[M]$.
    It is clear that each vertex and each edge of the graph appears in some bag, as each bag $B''_{t^M}$ contains $B_{t} \cap M$.
    Now, let us argue that for each vertex of $M$, the nodes whose bags contain the vertex form a connected set.
    This is clear for all vertices, except possibly for vertices $u_y \in U_y \setminus U_{\mathrm m}$ added in the third step of the construction.
    However, \zcref{claim:role_of_Uy} says that these vertices belong to the bag $B_{t_y}$.
    Therefore, adding them to the bags on the path from $t^M_x$ to $t^M_y$ in $T^M$ does not violate connectedness.

\begin{claim}\label{lem:widetildeT^M}
    ${\alpha(\mathcal{T}^M) \leq 4\ell-4}$.
\end{claim}

\begin{subproof}   
    Consider an independent set $I$ contained in a bag $B''_{t^M}$ of $\mathcal{T}^M$.
    We distinguish several cases:
    \begin{enumerate}
        [label=(\arabic*)]
        \item $I \subseteq B_t$,
        \item $I$ contains a vertex from $U_{\mathrm m}$,            
        \item $x \in I$,
        \item $I$ contains a vertex from $U_y \setminus U_{\mathrm m}$.
    \end{enumerate}
    By the definition of the bag $B''_{t^M}$, these four cases are exhaustive, though not necessarily disjoint. 
    Furthermore, the last two cases can only occur if $t$ is on the path from $t_x$ to $t_y$ in $T$.    
    
    \medskip
    \noindent\emph{Case (1).} In this case the claim follows directly from the assumption that ${\alpha(\mathcal{T}) \leq 4\ell-4}$.

    \medskip
    \noindent\emph{Case (2).} 
    We have $I \subseteq M = N(r) \cup \NN(x) \cup W_y$.
    Let us bound the size of the intersection of $I$ with each of these three sets separately.
    We have ${\abs{I \cap N(r)} \leq 2\ell-2}$ due to the choice of~$r$.
    Let $u_{\mathrm m} \in I \cap U_{\mathrm m}$.
    As $u_{\mathrm m} \in I$, we have $I \cap \NN(u_{\mathrm m}) = \emptyset$ and thus, by the definition of movable vertices, we have
    \[
        \abs{I\cap \NN(x)} = \abs{I\cap (\NN(x) \setminus \NN(u_{\mathrm m}))} \leq \ell - 1\,.
    \]
    Finally, we have $\abs{I\cap W_y} \leq \ell - 1$ by~\zcref{claimcollection2}\ref{claim2-2} and the fact that $(x,y)$ is a bad pair.
    Summing up, we obtain
    \[
        \abs{I} = \abs{I \cap N(r)} + \abs{I \cap \NN(x)} + \abs{I \cap W_y} \leq (2\ell-2) + (\ell-1) + (\ell-1) = 4\ell - 4
        \,.
    \]

    For the last two cases, recall that we are considering nodes $t^M$ such that $t$ is on the path from $t_x$ to $t_y$ in $T$.
    
    \medskip
    \noindent\emph{Case (3).} As $x \in I$, we have $I \subseteq M\setminus N(x) \subseteq N(r) \cup W_y$.
    Similarly as in Case (2), we observe that ${\abs{I \cap N(r)} \leq 2\ell-2}$ by the choice of~$r$ and $\alpha(W_y) \setminus \NN(x)) \leq \ell - 1$ by~\zcref{claimcollection2}\ref{claim2-2} and $(x,y)$ being bad.
    Thus, we obtain
    \[
        \abs{I} = \abs{I \cap N(r)} + \abs{I \cap W_y} \leq (2\ell-2) + (\ell-1) = 3\ell-3 \,.
    \]

    \medskip
    \noindent\emph{Case (4).}  
    Let $u_y \in I \cap (U_y \setminus U_{\mathrm m})$.
    By \zcref{claimcollection2}\ref{claim2-6}, $u_y$ is complete to $W_x$.
    Thus, $I \subseteq M\setminus N(u_y) \subseteq N(r) \cup \NN(y)$. 
    
    If $I \cap W_{xy} = \emptyset$, we have $I \subseteq N(r) \cup W_y$, implying
    \[
        \abs{I} = \abs{I \cap N(r)} + \abs{I \cap W_y} \leq (2\ell-2) + (\ell-1) = 3\ell-3\,,
    \]
    as in Case~(3).
    
    Otherwise, there is a vertex $w_{xy} \in I \cap W_x$.
    Consider an arbitrary $w_{u_y} \in \NN(u_y) \setminus W_x$; such a vertex exists by the definition of the set~$U$.
    Note that~$w_{xy}$ is nonadjacent to~$u_y$, as both of these vertices belong to~$I$.  
    This implies that~$w_{xy}$ is adjacent to~$w_{u_y}$, as otherwise there is an induced~$P_5$ of the form ${w_{xy}, x, r, u_y, w_{u_y}}$; see \zcref{fig:wxy_neighbors_wuy}.
    \begin{figure}[htbp]
        \centering
        \begin{minipage}{.42\textwidth}
            \centering
            \includegraphics[page=16,scale=1.2]{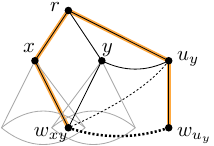}
            \captionof{figure}{The $P_5$ in the Case~(4) if $w_{xy}$ is not adjacent to $w_{u_y}$.}
            \label{fig:wxy_neighbors_wuy}
        \end{minipage}
        \qquad\quad~
        \begin{minipage}{.42\textwidth}
            \centering
            \includegraphics[page=9,scale=1.2]{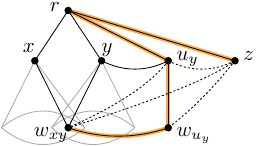}
            \captionof{figure}{The $P_5$  in the Case~(4)  if $I$ intersects $N(r) \setminus U$.}
            \label{fig:I_disjoint_from_N(r)_minus_U}
        \end{minipage}
    \end{figure}

    Now let us argue that $I$ is disjoint from ${N(r)\setminus U}$.         
    Suppose otherwise and let ${z \in I \cap (N(r) \setminus U)}$.
    Then we have a $P_5$ of the form $z,r,u_y,w_{u_y},w_{xy}$; see \zcref{fig:I_disjoint_from_N(r)_minus_U}.
    Note that all $z,u_y,w_{xy}$ belong in $I$ (as we assume), so they span no edge, and moreover~$z$ is not adjacent to~$w_{u_y}$, as ${\NN(z) \subseteq \NN(x) \cup \NN(y)}$ because~${z \notin U}$.

    \medskip
    The location of the set $I$ is now very restricted: since~$I$ is contained in $N(r) \cup \NN(y)$ and   disjoint from $N(r)\setminus U$, we have $I \subseteq U \cup \NN(y)$.
    Now we restrict the location even further.
    First, we may assume that Case~(2) does not apply, hence $I \cap U_{\mathrm m} = \emptyset$.
    In particular, this implies $I \cap U_0 = \emptyset$ by \zcref{cl:U0new}.
    Moreover, using \zcref{claimcollection2}\ref{claim2-1}, we have $I \cap U_x = \emptyset$.
    Therefore, we have
    \begin{equation}
        I \subseteq U_{xy} \cup (U_y \setminus U_{\mathrm m}) \cup \NN(y) \subseteq N(y)\,. \label{eq:Icontained-Case4}
    \end{equation}
    
    Recall that $I\subseteq B''_{t^M}$ for some node $t$ on the path from $t_x$ to $t_y$ in $T$.
    We complete the proof of Case~(4) in two steps: first, by showing that $I\subseteq  B''_{t_y^M}$ and second, by showing that $I \subseteq B_{t_y}$. 
    
    The first inclusion is clear for $t = t_y$, so we may assume that $t\neq t_y$.
    Note that \[B''_{t^M} = (B_t \cap M)\cup U_y \cup U _{\mathrm m} \cup \{x\}\,.\] 
    Since $U_y \cup U _{\mathrm m} \cup \{x\} \subseteq B''_{t^M_y}$, it is sufficient to argue that $I \cap B_t \cap M \subseteq B_{t_y}$, since $B_{t_y}\cap M \subseteq B''_{t^M_y}$.
    Consider any vertex $v \in I \cap B_t \cap M$ and recall that $v \in N(y)$ by \eqref{eq:Icontained-Case4}.
    As $T(v)$ contains $t$ and intersects $T(y)$, we must have $v \in B_{t_y}$ since for all nodes~$t'$ on the unique path between~$t$ and~$t_y$ we have that~${v \notin B_{t'}}$.
    This, together with the considerations above, shows that $I \subseteq B''_{t^M_y}$.
    
    It remains to show that $I \subseteq B_{t_y}$. 
    Again we will use \eqref{eq:Icontained-Case4}.
    The vertices of $U_{xy}$ are adjacent to both $x$ and $y$, so in the tree-decomposition $\mathcal{T}$ the subtree containing any fixed vertex from this set contains both $t_x$ and $t_y$, which in particular implies $U_{xy} \subseteq B_{t_y}$.
    By \zcref{claim:role_of_Uy} we immediately get $U_y \setminus U_{\mathrm m} \subseteq B_{t_y}$.
    Lastly, since every vertex added in the second and third step of the construction of the bag $B''_{t^M_y}$ belongs to ${U \cup \{x\}}$, no such added vertex belongs to $\NN(y)$.
    Hence, we have that $B''_{t_y^M} \cap \NN(y) = B_{t_y} \cap \NN(y)$.
    Together with $I \subseteq B''_{t^M_y}$, this implies that $I\cap  \NN(y)\subseteq B''_{t_y^M} \cap \NN(y) \subseteq B_{t_y}$. 
    Putting these together with \eqref{eq:Icontained-Case4}, we obtain that  $I \subseteq B_{t_y}$, therefore, ${\abs{I} \leq \alpha(\mathcal{T}) \leq 4\ell-4}$.

    This completes the proof of both Case (4) and \zcref{lem:widetildeT^M}.
\end{subproof}

Before proceeding to the next step, let us prove one more claim that will be useful later.

\begin{claim}\label{claim:widetildeB_t_x_contains_U0_Uy_NxcapNy}
	The bag ${B}''_{t^M_x}$ contains $U_0 \cup U_y \cup (N(x) \cap N(y))$. 
\end{claim}

\begin{subproof}
    As $\mathcal{T}$ is a tree-decomposition of $G'$, for every vertex $v$ that is adjacent to $x$ and $y$, the subtree $T(v)$ spans both $t_x$ and $t_y$, so $N(x) \cap N(y)  \subseteq B_{t_x} \subseteq B''_{t^M_x}$.
	Moreover, $B''_{t^M_x}$ contains $U_0$ and $U_y$ by \zcref{cl:U0new} and the second and third step in the definition of $\mathcal{T}^M$.
\end{subproof}

\paragraph{Defining bags $\{B'_t\}_{t \in V(T^M)}$}
From now on, we use the notation $T^M(v)$ with respect to these final bags, i.e., by $T^M(v)$ for some vertex $v$, we mean the set of nodes $t^M \in V(T^M)$ such that $v \in B'_{t^M}$.

\begin{itemize}
    \item Let $c \in V(G')\setminus M$ be a vertex whose neighborhood in $M$ is not contained in a single bag of $\mathcal{T}^M$.
    We add $c$ to the bags of the minimum number of nodes so that $T^M(c)$ is a connected subtree and each vertex of $N(c)\cap M$ is contained in the bag of some node of $T^M(c)$.
\end{itemize}

In other words, for every $a,b \in N(c) \cap M$ with disjoint $\{t \in V(T^M) \mid a \in B''_{t}\}$ and $\{t \in V(T^M) \mid b \in B''_{t}\}$, we add $c$ to all the bags of all nodes on the path between $\{t \in V(T^M) \mid a \in B''_{t}\}$ and $\{t \in V(T^M) \mid b \in B''_{t}\}$.

Effectively, this partitions each component $C$ of $G' - M$ into the set $C'$ of vertices of $C$ that appear in $\bigcup_{t^M \in V(T^M)} B'_{t^M}$ (that is, the vertices of $C$ that were added to the bags in the above step), and the rest, $C'' \coloneqq V(C) \setminus C'$.
We define $M' \coloneqq M \cup \bigcup_{C \in \cc} C'$.

The main goal of this part is to show that ${\mathcal{T}}^{M'} \coloneqq (T^M,\{B'_t\}_{t \in V(T^M)})$ is a tree-decomposition of $G'[M']$ with independence number at most~${4\ell-4}$, that additionally satisfies some useful properties.
Namely, the crucial observation will be \zcref{cl:connectedcomponentscommonbag} that shows that the neighborhood of each component $D$ of $G'-M'$ is contained in a bag of a single node.
This node will be the natural attachment point for the tree-decomposition of $D$ constructed in the next section.

Let us start with two simple observations.

\begin{claim}\label{obs:trees_in_T^M_vs_T}
    For each $c \in \bigcup_{C \in \cc} C'$ and each $t \in V(T)$, it holds that if $c \in B'_{t^M}$, then $c \in B_{t}$.
\end{claim}

\begin{subproof}
    Note that for any $c \in M' \setminus M$ and any $a \in M$ adjacent to $c$, the trees $T(a)$ and $T(c)$ intersect.
    Furthermore, for every $t \in V(T)$ we have 
    \[
        B_{t} \cap M \subseteq B''_{t^M} = B'_{t^M} \cap M\,.
    \]
    This implies that for every $t \in V(T)$ such that $t\in T(a)$, we have $t^M\in T^M(a)$.
     In particular, since the trees $T(a)$ and $T(c)$ intersect, this implies that $\{t^M \mid t \in T(c) \}$ intersects $T^M(a)$.

    Since $c$ was added to the bags of the minimum set of nodes that meets $T^M(a)$ for every $a\in N(c) \cap M$, and $T^M$ is a tree, we conclude that $T^M(c) \subseteq \{t^M \mid t \in T(c) \}$.
\end{subproof}

\begin{claim}\label{obs:extended_bags_in_T^M_left_of_x}
    For a node $t \in V(T)$, if $B'_{t^M} \neq B''_{t^M}$,
    then either $t = t_x$ or $t$ belongs to a component of $T - t_x$ that does not contain $t_y$.
\end{claim}
\begin{subproof}
    Let $T_y$ be the component of $T - t_x$ that contains $t_y$ and let $T^M_y$ be the corresponding subgraph of $T^M$.
    Let $t \neq t_x$ be a node of $T$ such that $B'_{t^M} \neq B''_{t^M}$.    
    This means that there is some component $C$ of $G'-M$ and a vertex $c$ of $C$ such that $c \in B'_{t^M} \setminus B''_{t^M}$.

    Let us analyze which bags contain neighbors of $c$.    
	Recall that by \zcref{claim:widetildeB_t_x_contains_U0_Uy_NxcapNy} we have $U_0 \cup U_y \cup (N(x)\cap N(y)) \subseteq {B}''_{t^M_x}$.
    Due to \zcref{claimcollection1}\ref{claim1-1}, we have $N(C) \subseteq U \cup W_{xy}$.
    Since $U_{xy} \subseteq N(x) \cap N(y) \subseteq {B}''_{t^M_x}$, we conclude that 
    the only possible neighbors of $c$ outside of ${B}''_{t^M_x}$ belong to $U_x \setminus B''_{t^M_x}$.
    Note that in particular they are in $U_x \setminus U_{\mathrm m}$,
    as $U_{\mathrm m}$ is contained in $B'_{t^M_x} \subseteq B''_{t^M_x}$.    
    
    Each $u_x \in U_x$ is adjacent to $x$ in $G'$, so $T(u_x) \cap T(x) \neq \emptyset$.
    Recall that $t_x$ was chosen to be the node of $T(x)$ closest to $T(y)$, so $T(x)$ does not intersect $T_y$.
    Thus, $T(u_x) \cap T(x)$ also does not intersect $T_y$.
    Moreover, if $T(u_x)$ does not contain $t_x$, then  $T(u_x)$ is contained in one of the components of $T - t_x$ other than $T_y$.
    
    Recall that in the construction of the tree-decomposition $\mathcal{T}^M$,
    none of the vertices $u_x \in U_x \setminus U_{\mathrm m}$ was included in the bags of any node except for their corresponding nodes of $\mathcal{T}$.    
    Thus, the nodes of $T^M(u_x)$ are precisely the copies of nodes of $T(u_x)$.
    In particular, $T(u_x)$ does not contain $t_x$ if and only if $T^M(u_x)$ does not contain $t^M_x$. 
    
    Hence, for each $u_x \in U_x \setminus B''_{t^M_x}$, the set $T^M(u_x)$ is contained in a component of $T^M - t^M_x$ distinct from $T^M_y$.
    Consequently, we have that $\bigcup_{v \in N(c) \cap M} T^M(v) \subseteq T^M - T^M_y$.
    Since $T^M - T^M_y$ is connected, we never added $c$ to the bag of any node in $T^M_y$.
\end{subproof}

Now we may proceed to \zcref{cl:connectedcomponentscommonbag} that we already announced.
\begin{claim}\label{cl:connectedcomponentscommonbag}
    Let $C$ be a component of $G' - M$.
    Then, there is a node $t^M_C \in V(T^M)$ such that $C' \subseteq B'_{t^M_C}$ and $N(C'') \cap M \subseteq B'_{t^M_C}$.
\end{claim}
\begin{subproof}    
    We first consider the case that $N(C) \setminus U_{xy}$ is not contained in a bag of a single node of ${\mathcal{T}}^{M'}$.
    Note that by the Helly property of trees this implies that there are vertices $a,b \in N(C) \setminus U_{xy}$
    such that $T^M(a)$ and $T^M(b)$ are disjoint.
    
    Recall that by \zcref{claimcollection1}\ref{claim1-1}, $N(C)$ is contained in $U \cup W_{xy}$.
    Now, combining \zcref{claimcollection1}\ref{claim1-1} and \zcref{claimcollection2}\ref{claim2-7} 
    we obtain that $C$ is complete to $N(C) \setminus U_{xy}$.    
    Note that this in particular implies that that $C = C'$.
    As every $c \in C$ is adjacent to both $a$ and~$b$, we observe that $T^M(c)$ contains the path between $T^M(a)$ and $T^M(b)$ in $T^M$.
    Hence, there is a node $t^M_C$ on that path contained in $T^M(c)$ for every $c \in C'$.
    Furthermore, the inclusion $N(C'') \cap M \subseteq B'_{t^M_C}$ holds vacuously as $C'' = C \setminus C' = \emptyset$.

\medskip
    Now suppose that $N(C) \setminus U_{xy}$ is contained in a bag of a single node of ${\mathcal{T}}^{M'}$ (in particular, this set might be empty).
    Let $t^M_C$ be such a node of $T^M$ that is closest to $t^M_x$. 
   
    Let us first observe that it is enough to show that every vertex in $C$ with a neighbor in $M \setminus B'_{t^M_C}$ belongs to both $C'$ and $B'_{t^M_C}$.
    Indeed, since every vertex in $C'$ has a neighbor in $M \setminus B'_{t^M_C}$, the inclusion $C' \subseteq B'_{t^M_C}$ would follow immediately.
    Furthermore, the inclusion $N(C'') \cap M\subseteq B'_{t^M_C}$ is equivalent to the claim that the neighborhood in $M$ of every vertex in $C''$ is contained in the bag $B'_{t^M_C}$, which is equivalent to the claim that every vertex in $C$ with a neighbor in $M \setminus B'_{t^M_C}$ belongs to $C'$.
    
    So let $c\in C$ be a vertex with a neighbor in $M \setminus B'_{t^M_C}$.
    Such a neighbor must belong to $U_{xy} \setminus B^M_{t_C}$; let us call it $u_{xy}$.
    Recall that due to \zcref{claim:widetildeB_t_x_contains_U0_Uy_NxcapNy} we have $U_{xy} \subseteq N(x) \cap N(y) \subseteq B''_{t^M_x} \subseteq B'_{t^M_x}$.
    Note that this in particular implies that $t^M_C \neq t^M_x$.
    We claim the following holds.

    \begin{description}
        \item[$(\star)$] The vertex $c$ has a neighbor $u_x \in U_x$ that belongs to $B'_{t^M_C}$, but to no other bag on the path from $t^M_C$ to $t^M_x$ in $T^M$.
    \end{description}
       
        Since $t^M_C \neq t^M_x$, there is a vertex $u_x \in N(C) \setminus U_{xy}$ that did not allow for choosing $t^M_C$ to be closer to $t^M_x$, so $u_x \in B'_{t^M_C}$ but~$u_x$ does not belong to the bag of any other node on the path from $t^M_C$ to $t^M_x$. 
        In particular, by \zcref{claimcollection1}\ref{claim1-1} and \zcref{claim:widetildeB_t_x_contains_U0_Uy_NxcapNy}
        the only vertices from $N(C) \setminus U_{xy}$ that might be absent from $B'_{t^M_x}$  are in $U_x$. Thus, $u_x \in U_x$.        
        Since $C$ is complete to $N(C) \cap U_x$ by \zcref{claimcollection1}\ref{claim1-2}, we have that $c$ is adjacent to $u_x$, completing the proof of $(\star)$. 
                
    \medskip
    Since~$u_{xy} \in B'_{t_x^M} \setminus B'_{t_C^M}$ and~$u_x \in B'_{t_C^M} \setminus B'_{t^M}$ for any node~$t^M\neq t^M_C$ on  the path from $t^M_C$ to $t^M_x$ by $(\star)$, the vertices~$u_{xy}$ and~$u_x$ are not contained in the same bag for any node of~$T^M$. 
    Since~$u_{xy}, u_x \in N(c) \cap M$, it follows that~${c \in C'}$. 
    In particular, the subtree $T^M(c)$ contains the node $t^M_C$ as it is the closest node of $T^M(u_x)$ to $T^M(u_{xy})$, so~${c \in B'_{t^M_C}}$, as required.
\end{subproof}

Now we are ready to prove that ${\mathcal{T}}^{M'} = (T^M,\{B'_t\}_{t \in V(T^M)})$ is a tree-decomposition of $G'[M']$.

\begin{claim}
    ${\mathcal{T}}^{M'}$ is a tree-decomposition of $G'[M']$.
\end{claim}
\begin{subproof}
    Notice that for every $v \in M$ and every $t^M \in V(T^M)$ it holds that $v \in B'_{t^M}$ if and only if  $v \in B''_{t^M}$.
    This means that $T^M(v)$ is nonempty and connected as $\mathcal{T}^M$ is a tree-decomposition of $G'[M]$.

    For each $c \in M' \setminus M$, the set $T^M(c)$ is nonempty, as $M' \setminus M$ are precisely the vertices that appear in some bag.
    Furthermore, connectedness of $T^M(c)$ follows directly from the definition of the bags.

    Now consider an edge $e$ of $G'[M']$.
    If both endpoints of $e$ are in $M$, then it is contained in $B''_{t^M}$ for some $t^M$ and thus,
    also in $B'_{t^M}$.
    If $e$ is an edge between $M$ and some component $C$ of $G'-M$, then the endpoint $c$ of $e$ that belongs to $C$ also belongs to $C'$, which implies that both endpoints of $e$ appear in some bag of $T^M(c)$.
    Finally, consider an edge joining two vertices from $M' \setminus M$.
    This means that there is a component $C$ of $G'-M$ such that both endpoints of $e$ belong to $C'$. 
    By \zcref{cl:connectedcomponentscommonbag} there is a node~$t^M_C$ of~$T^M$ such that both endpoints of $e$ are contained in $B'_{t^M_C}$. 
\end{subproof}

\begin{claim}
    \label{claim:vertices_in_proof_of_T^M}
    Suppose some bag~$B'_{t^M}$ of~$\mathcal{T}^{M'}$ contains an independent set~$I$ of size more than~${4\ell-4}$. 
    Then there exists a component $C$ of $G'-M$ with a vertex $c\in C'$ that has a neighbor $u_x \in U_x \setminus U_{\mathrm m}$ such that some vertex $u_{\mathrm m} \in U_{\mathrm m}$ is adjacent to neither $c$ nor $u_x$.
\end{claim}

\begin{subproof}
	Observe that $I$ needs to contain a vertex $c \in \bigcup_{C \in \cc} C'$ that was added to~$B'_{t^M}$ when constructing the decomposition $\mathcal{T}^{M'}$ from $\mathcal{T}^M$ because ${\alpha(\mathcal{T}^M) \leq 4\ell-4}$ by \zcref{lem:widetildeT^M}. 
    Since, in particular, $B'_{t^M} \neq B''_{t^M}$, by \zcref{obs:extended_bags_in_T^M_left_of_x} we observe that either~${t = t_x}$ or~${t}$ belongs to a component of~$T - t_x$ that does not contain~$t_y$. 
    Note that, since $N(x)\cup N(y)\subseteq M$ and $c\not\in M$, vertex $c$ is adjacent to neither $x$ nor $y$.
   
    By \zcref{obs:trees_in_T^M_vs_T}, each vertex in~${B'_{t^M} \setminus B''_{t^M}}$ is contained in~$B_t$.
    This implies that $c\in B_t$ and, since~${\alpha(\mathcal{T}) \leq 4\ell-4}$ by our original assumption, it also implies that $I$ contains a vertex $u \in U_{\mathrm m} \cup U_y \cup \{x\}$ that was added to~$B''_{t^M}$ when constructing the decomposition $\mathcal{T}^M$ from $\mathcal{T}$. 
    In particular, $u\notin B_t$.
    Note that vertices $c$ and $u$ are nonadjacent, since they both belong to~$I$. 
    
    Next, we show that ${u \in U_{\mathrm{m}}}$. 
    Suppose for a contradiction that~${u \notin U_{\mathrm{m}}}$.
    Then,~$u$ belongs to the set ${(U_y \setminus U_{\mathrm{m}}) \cup \{x\}}$. 
    Hence, by construction of $\mathcal{T}^M$, vertex $u$ was added to $B''_{t^M}$ for all the nodes on the path between $t_x$ and~$t_y$ in~$T$, so we conclude from \zcref{obs:extended_bags_in_T^M_left_of_x} that~$t = t_x$. 
    This implies that $u\neq x$, since~${u \notin B_{t_x}}$ and~${x \in B_{t_x}}$.
    Hence,~${u \in U_y \setminus U_{\mathrm{m}}}$.
    By \zcref{claim:role_of_Uy}, $u \in B_{t_y}$. 
    Since~$u \notin B_{t_x}$, we have that~$T(u)$ is a subtree of the component of~$T - t_x$ that contains $t_y$. 
    Moreover, using the fact $c\in B_{t_x}$ and applying   \zcref{obs:extended_bags_in_T^M_left_of_x} to the nodes of $T(c)$, we obtain that the subtree~$T(c)$ contains~$t_x$ but no vertex from the component of~$T - t_x$ containing~$t_y$. 
    Hence, $T(u)$ and~$T(c)$ are disjoint. 
   Since~${c \in B'_{t^M_x} \setminus B''_{t^M_x}}$, there exists a vertex~${u_x \in (N(c) \cap M) \setminus B''_{t^M_x}}$. 
    Since $B_{t_x}\subseteq B''_{t^M_x}$, we have that~${u_x \notin B_{t_x}}$.
    Furthermore, by 
    \zcref{claimcollection1}\ref{claim1-1} and \zcref{claim:widetildeB_t_x_contains_U0_Uy_NxcapNy}, ${u_x \in U_x}$; hence, $u_x$ is adjacent to $x$ but not to $y$. 
    Since~$T(u_x)$ intersects~$T(c)$ and does not contain $t_x$, we infer that $T(u_x)$ is a subtree of a component of~$T-t_x$ that does not contain~$t_y$.
    This implies that~$T(u_x)$ and~$T(u)$ are disjoint; in particular,~$u_x$ is not adjacent to $u$. 
    Let~${z \in W_x}$. 
    Note that~$z$ exists since~$(x,y)$ is a bad pair.
    Furthermore, $z$ is adjacent to $x$ but not to $y$, and $z$ is adjacent to $u$ (by \zcref{claimcollection2}\ref{claim2-6}) but not to $c$ (by \zcref{claimcollection1}\ref{claim1-1} again).
    Hence, we have a $P_5$ either of the form $c, u_x, z, u, y$ (if $z$ is adjacent to $u_x$) or of the form $c, u_x, x, z, u$ (otherwise).
    Both cases yield a contradiction, so~${u \in U_m}$. 
    From now on, we denote~$u$ by~$u_{\mathrm{m}}$. 
    
    Since we added~$c$ to $B'_{t^M}$, by construction and \zcref{claimcollection1}\ref{claim1-1}, there are neighbors $a, b \in N(c) \cap M \subseteq U \cup W_{xy}$ that are not contained in a common bag of $\mathcal{T}^M$ and, moreover,~$t$ lies on the path in~$T$ between nodes~$t_a$ and~$t_b$  with~${a \in B''_{t^M_a}}$ and~${b \in B''_{t^M_b}}$.
    Since vertices of $U_{\mathrm m}$ belong to all bags of $\mathcal{T}^M$, it follows that $a, b \notin U_{\mathrm m}$ and hence, by \zcref{cl:U0new}, $a, b \notin U_0$. 
    Therefore, $a, b \in U_x \cup U_y \cup U_{xy} \cup W_{xy}$. 
    Note that at least one of $a$ and~$b$ does not belong to any bag of a node from the component of $T - t$ containing $t_y$, as otherwise the path in $T$ between $t_a$ and $t_b$ would not go through $t$.
    Without loss of generality, say this is the case for~$a$. 
    Now, since either~${t = t_x}$ or~${t}$ belongs to a component of~$T - t_x$ that does not contain~$t_y$, we observe that~$y \notin B''_{t^M}$ and no bag of $\mathcal{T}^M$ contains both $a$ and $y$, implying that $a$ and $y$ are nonadjacent in $G'$. 
    Consequently, the only remaining option is that $a \in U_x \setminus U_{\mathrm m}$. 
    If $a$ is not adjacent to~$u_{\mathrm m}$, we take $u_x \coloneqq a$; this way, $c$ is adjacent to~$u_x$ by the definition of $c$, and $c$ is nonadjacent to $u_{\mathrm m}$ as $c,u_{\mathrm m} \in I$, and we are done. 
    So we may assume that $a$ is adjacent to~$u_{\mathrm m}$.

	We next show that~$u_{\mathrm m}$ is not adjacent to~$y$. 
    Let us assume for a contradiction that $u_{\mathrm m} y \in E(G')$. 
    Since $u_{\mathrm m} \notin B_t$, the tree $T(u_{\mathrm m})$ is a subtree of the component of $T - t$ that contains $t_y$, since $u_{\mathrm m} y \in E(G')$. 
    Therefore, as before, we have that the tree $T(u_{\mathrm m})$ is disjoint from $T(a)$, so $u_{\mathrm m} a \notin E(G')$, a contradiction. 

    We thus have $u_{\mathrm m} \in U_{\mathrm m} \setminus U_y$.
	As $u_{\mathrm m} \notin B_t$, we know that $T(u_{\mathrm m})$ is contained in one component of $T - t$.
    Recall also that $T(a)$ does not intersect the component of $T - t$ containing~$t_y$; in particular, $a$ is not adjacent to $y$.
    Since $a$ is adjacent to~$u_{\mathrm m}$, the component of $T - t$ containing $T(u_{\mathrm m})$ intersects $T(a)$.
    Hence, $T(b)$ and $T(u_{\mathrm m})$ are disjoint, since otherwise the path in $T$ between $t_a$ and $t_b$ would not go through $t$. 
    Consequently,~$u_{\mathrm m}$ and~$b$ are nonadjacent. 
    
    To be able to take~$u_x$ to be~$b$, all that is left to argue is that $b \in U_x \setminus U_{\mathrm m}$.
    Suppose for a contradiction that~$b \notin U_x \setminus U_{\mathrm{m}}$.
    Since $b \notin U_{\mathrm m}$, it follows that $b \notin U_x$.
    Then either $b \in U_y \cup U_{xy}$, or $b \in W_{xy}$.
    In both cases, $b$ is adjacent to $y$, and hence, we have a $P_5$ of the form $u_{\mathrm m},a,c,b,y$, a contradiction.
    We conclude that $b \in U_x \setminus U_{\mathrm m}$ and we may take $u_x \coloneqq b$. 
\end{subproof}

We now have everything ready to prove that the independence number of ${\mathcal{T}}^{M'}$ is at most~${4\ell-4}$.

\begin{claim}\label{lem:T^M}
   ${\alpha(\mathcal{T}^{M'}) \leq 4\ell-4}$.
\end{claim}

\begin{subproof}    
    Consider an arbitrary $t^M \in V(T^M)$.
    Suppose for a contradiction that there exists an independent set $I \subseteq B'_{t^M}$ of size more than~${4\ell-4}$. 
    Let~$C$, $c$, $u_x$, and~$u_{\mathrm{m}}$ be as in \zcref{claim:vertices_in_proof_of_T^M}.  

    We first show that $\NN(x)\setminus \NN(u_x)\subseteq \NN(y)$.
    Indeed, suppose that this is not the case. 
    Then, there exists a vertex $w \in \NN(x) \setminus \NN(u_x)$ that is not adjacent to~$y$.   
    As $u_{\mathrm m} \in U_{\mathrm m}$, we have that ${\NN(x)\setminus \NN(u_{\mathrm m})\subseteq \NN(y)}$.
        Therefore, $w$ is adjacent to $u_{\mathrm m}$.
        Since $c$ is not adjacent to $w$ by \zcref{claimcollection1}\ref{claim1-1}, there is a copy of $P_5$ on vertices $c,u_x,r,u_{\mathrm m},w$ (see \zcref{fig:u_x_complete_to_N_2(x)setminusN_2(y)}), a contradiction.
        Hence, $\NN(x)\setminus \NN(u_x)\subseteq \NN(y)$, as claimed.
    
    \begin{figure}[htbp]
        \centering
        \begin{minipage}{.48\textwidth}
            \centering
            \includegraphics[page=10,scale=1.2]{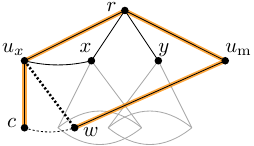}
            \captionof{figure}{The $P_5$ when \newline $\NN(x)\setminus \NN(u_x)\nsubseteq \NN(y)$.}
            \label{fig:u_x_complete_to_N_2(x)setminusN_2(y)}
        \end{minipage}
        \begin{minipage}{.48\textwidth}
            \centering
            \includegraphics[page=11,scale=1.2]{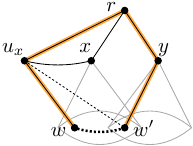}
            \captionof{figure}{The $P_5$ when $w$ is nonadjacent to $w'$.}
            \label{fig:wuxrywp_path_on_nonincident_wwp}
        \end{minipage}
    \end{figure}
    
    Since $u_x \notin U_{\mathrm m}$ and $\NN(x)\setminus \NN(u_x)\subseteq \NN(y)$, this implies that $\alpha(\NN(x) \setminus \NN(u_x)) \geq \ell$.
    Moreover, since $\NN(x)\setminus \NN(u_x)\subseteq \NN(y)$, we obtain that $\NN(x) \setminus \NN(u_x) = W_{xy} \setminus \NN(u_x)$ and hence, $\alpha(W_{xy} \setminus \NN(u_x)) \geq \ell$.
    The sets $W_x$ and $W_{xy} \setminus \NN(u_x)$ are disjoint and each contains an independent set of size $\ell$.
    Hence, they cannot be complete to each other, as that would imply the presence of a copy of~$K_{\ell, \ell}$ in $G$.
    Thus, there exists a pair of nonadjacent vertices ${w \in W_x}$ and ${w' \in W_{xy} \setminus \NN(u_x)}$.
    But now, $G$ contains a copy of $P_5$ on vertices $w,u_x,r,y,w'$ (see \zcref{fig:wuxrywp_path_on_nonincident_wwp}), a contradiction.
\end{subproof}

\paragraph{The tree-decompositions ${\mathcal{T}^D}$}
Let $D$ be a component of $G' - M'$.
In this part, we define the tree-decompositions $\mathcal{T}^D$ and prove the necessary properties for a smooth construction of the final decomposition for Case~2.

We define the auxiliary graph~$H_D$ with vertex set~$N[D]$ whose edge set consists of all edges that have at least one endpoint in~$D$. 
Note that $H_D$ is not necessarily an induced subgraph of $G'$.
We construct a tree-decomposition~${\mathcal{T}^D = (T^D, 
\{B'_{t}\}_{t \in V(T^D)})}$ of~$H_D$ as follows. 
We let $T^D$ be a copy of $T$, with a node $t^D$ corresponding to each node $t$ of $T$, and, for all $t\in V(T)$, we define  
\[
    B'_{t^D} \coloneqq (B_t \cap N[D]) \cup (N(D) \setminus U_{xy}) \,.
\]

Clearly, each edge of~$H_D$ is contained in some bag~$B_t$ of~$\mathcal{T}$, and hence, in~${B_t \cap N[D]}$. 
So~$\mathcal{T}^D$ is indeed a tree-decomposition of~$H_D$, since~$\mathcal{T}$ is a tree-decomposition of~$G'$ and any vertex added to some bag was added to all bags. 
Moreover, since~$N(x) \cap N(y) \subseteq B_{t_x}$, as observed in the proof of \zcref{claim:widetildeB_t_x_contains_U0_Uy_NxcapNy}, it follows that  $N(D) \cap U_{xy}\subseteq B_{t_x}$ and, hence, 
\begin{equation}\label{eq:T^D}
    N(D) \subseteq B'_{t^D_x} \,.
\end{equation}

    Let $C$ be the component of $G' - M$ that contains $D$. 
    As before, let~$C'$ be the set of vertices of~$C$ that appear in $\bigcup_{t^M \in V(T^M)} B'_{t^M}$ and let~$C'' = C \setminus C'$. 
    By \zcref{cl:connectedcomponentscommonbag}, there is a node $t^M_C \in V(T^M)$ such that $C' \subseteq B'_{t^M_C}$ and $N(C'') \cap M \subseteq B'_{t^M_C}$.
    Note that~$C' = C\cap M'$ and, consequently,~$C'' = C\setminus M'$. 
    Furthermore, since $D$ is disjoint from $M'$, we have $D\subseteq C''$.
    
\begin{claim}\label{claim:D_complete_to_neighborhood_in_C'}
$D$ is complete to $N(D) \cap C$. 
\end{claim}
\begin{subproof}
    \begin{figure}[htbp]
        \centering
        \includegraphics[page=15,scale=1.2]{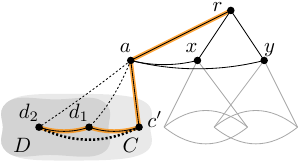}
        \caption{The $P_5$ if $D$ is not complete to $N(D) \cap C'$.}
        \label{fig:p5_dcomponent_incomplete}
    \end{figure}

    Suppose for a contradiction that~$D$ is not complete to~$N(D) \cap C$.
    Then there are vertices $c' \in C'$ and $d_1, d_2 \in D$ such that $c$ is adjacent to~$d_1$ but not to~$d_2$, and $d_1$ and~$d_2$ are adjacent (similarly as in \zcref{claimcollection1}\ref{claim1-2}). 
    Since $N(C'') \cap M \subseteq B'_{t^M_C}$ and $d_1,d_2\in C''$, all the neighbors of  $d_1$ and $d_2$ in~$M$ belong to $B'_{t^M_C}$.
    However, as $c'$ belongs to $C'$, it has a neighbor~${a \in M}$ that does not belong to $B'_{t^M_C}$. 
    Hence, $a$ is adjacent to neither~$d_1$ nor~$d_2$. 
    Since~$a$ is not complete to~$C$, by 
    \zcref{claimcollection1}\ref{claim1-1} and~\ref{claim1-2}, and \zcref{claimcollection2}\ref{claim2-7}, we infer that $a \in U_{xy}$.
    But then we have a $P_5$ of the form $d_2,d_1,c',a,r$ (see \zcref{fig:p5_dcomponent_incomplete}).
\end{subproof}

\begin{claim}
    \label{claim:alpha-T^D}
    ${\alpha(G'[B'_{t^D}]) \leq 4\ell-4}$ for every~$t^D \in V(T^D)$. 
\end{claim}

\begin{subproof}
    Consider an independent set $I$ in an arbitrary bag $B'_{t^D}$. 
    If ${I \subseteq B_t \cap N[D]}$, then clearly ${\abs{I} \leq \alpha(\mathcal{T}) \leq 4\ell-4}$. 
    So we may assume that~$I$ contains a vertex from $N(D) \setminus U_{xy}$. 
    Since these vertices are complete to $D$ by 
    \zcref{claimcollection1}\ref{claim1-1} and~\ref{claim1-2}, \zcref{claimcollection2}\ref{claim2-7}, and \zcref{claim:D_complete_to_neighborhood_in_C'}, we have that $I \subseteq N(D)$.
    However, since $D\subseteq C''$, it follows that $N(D) \subseteq (N(C'')\cap M)\cup C'$ and, consequently, $N(D) \subseteq B^M_{t_C}$, implying that $I \subseteq B^M_{t_C}$.
    Then, by \zcref{lem:T^M}, we obtain that ${\abs{I} \leq 4\ell-4}$. 
\end{subproof}

\paragraph{Putting everything together}
Finally, we summarize the results from above to derive a final contradiction that will finish the proof of \zcref{lem:bagwithNr}.

Consider the tree-decompositions $\mathcal{T}^{M'}$ and $\mathcal{T}^D$ for each component $D$ of $G' - M'$.
We construct a tree~$T'$ from the union of the tree~$T^M$ with all the trees~$T^D$, by joining, for each~$D$, the node $t^D_x$ of~$T^D$ to the node ${t^M_C}$ of $T^M$, where $C$ is the component of $G' - M$ such that $D \subseteq C$ and $t^M_C$ is the node from \zcref{cl:connectedcomponentscommonbag}. 
Clearly, $T'$ is a tree. 

Recall that for each node~$s$ of~$T'$, we already defined a set $B'_s\subseteq V(G')$ as the respective bag of either~$\mathcal{T}^{M'}$ or~$\mathcal{T}^D$. 
We claim that $\mathcal{T}' = (T', \{ B'_s \}_{s \in V(T')})$ is a tree-decomposition of~$G'$ with ${\alpha(\mathcal{T}') \leq 4\ell-4}$ and the property that the number of pairs of neighbors of~$r$ that share a bag is increased compared to $\mathcal{T}$. 

First let us observe that~$\mathcal{T}'$ is indeed a tree-decomposition. 
Since each vertex and each edge of~$G'$ is contained in either~$G'[M']$ or in~$H_D$ for some component~$D$ of~$G' - M'$, the union of all bags is equal to~$V(G')$ and each pair of adjacent vertices is contained in some common bag. 
For each vertex~${v \in D}$ for a component~$D$ of~$G' - M'$, if~$v \in B'_s$ for some $s\in V(T')$, then~${s \in V(T^D)}$, hence~$T'(v)$ is connected, since $\mathcal{T}^D$ is a tree-decomposition. 
Similarly, for each vertex~${v \in M}$, if~$v \in B'_s$ for some $s\in V(T')$, then~${s \in V(T^M)}$, so~$T'(v)$ is connected, since $\mathcal{T}^{M'}$ is a tree-decomposition. 
So consider a vertex~${v \in M' \setminus M}$. 
Suppose~$v$ is contained in a bag of~$\mathcal{T}^D$ for some component~$D$ of~${G' - M'}$, and let~$C$ be the component of~$G' - M$ containing~$D$.
Since $v\in N(D)$, we obtain by \eqref{eq:T^D} that $t_x^D \in T'(v)$, and since $N(D) \subseteq (N(C'')\cap M)\cup C'\subseteq B'_{t^M_C}$ (as observed in the proof of \zcref{claim:alpha-T^D} and by the choice of $t_C^M$), we obtain that $t_C^M \in T'(v)$. 
Since by construction, $t_x^D t_C^M \in E(T')$, it follows that~$T'(v)$ is connected.  

Moreover, ${\alpha(\mathcal{T}') \leq 4\ell-4}$ follows immediately from the definition of~$\mathcal{T}'$ and \zcref{lem:T^M,claim:alpha-T^D}.

For the last property, observe first that if two neighbors of~$r$ are contained in a common bag~$B_t$ of~$\mathcal{T}$, then they are both in~$M'$, and by the construction of~$\mathcal{T}^M$,~$\mathcal{T}^{M'}$, and~$\mathcal{T}'$, they are in the bag~$B'_{t^M}$ of~$\mathcal{T}'$. 
Moreover, $x$ and~$y$ share the bag~$B'_{t_y^M}$, so the number of pairs of neighbors of~$r$ that share a bag is indeed increased compared to $\mathcal{T}$. 

Hence, the existence of~$\mathcal{T}'$ contradicts our initial assumption that~$\mathcal{P}$ is nonempty. 
Hence, $\mathcal{T}$ was the desired tree-decomposition of~${G'}$ all along. 
\hfill\qed

\section{Algorithmic consequences}

Let us conclude the paper by discussing some algorithmic consequences of our results.
First, we observe that the proof of \zcref{thm:P5treealpha} can be easily turned into a polynomial-time algorithm.
The only non-obvious part is that on the way we need to find largest independent sets in various induced subgraphs of the input graph,
e.g., to find the vertex~$r$ with small~${\alpha(N[r])}$ or to decide whether bad pairs exist.
However, finding a largest independent set in a $P_5$-free graph can be done in polynomial time~\cite{LokshantovVV14}.
Thus, we obtain the following algorithmic version of \zcref{thm:P5treealpha}.

\begin{theorem}[algorithmic version of \zcref{thm:P5treealpha}]
    \label{thm:alg}
    There is an algorithm that, given a $P_5$-free graph~$G$ and an integer~$\ell$,
    in time polynomial in~$\abs{V(G)}$ and~$\ell$ returns one of the following outputs:
    \begin{itemize}
        \item an induced~$K_{\ell,\ell}$ in~$G$, or
        \item a tree-decomposition of~$G$ with independence number at most~${4\ell-4}$.
    \end{itemize}    
\end{theorem}

In particular, the running time depends only polynomially on~$\ell$.
Thus, the algorithm yields a polynomial-time constant-factor-approximation algorithm for computing tree-independence number (and a corresponding decomposition) in $P_5$-free graphs.
Indeed, we run the algorithm for~${\ell=1,2,\ldots}$ and consider the smallest value for which the algorithm returns the second outcome, i.e., a decomposition of tree-independence number~${k^* \leq 4\ell-4}$.
As the input graph~$G$ contains~$K_{\ell-1,\ell-1}$ as an induced subgraph, we conclude that
\[
    \ell-1 \leq \tin(G) \leq k^* \leq 4\ell-4 \leq 4\tin(G).
\]

We note that for general graphs, the existence of essentially any non-trivial polynomial-time approximation algorithm for computing tree-independence number (or even any \textsf{FPT}-approximation, when parameterized by the tree-independence number) is excluded by known hardness results~\cite{MR4984710}.

Conversely, the algorithm from \zcref{thm:alg} can be seen as a polynomial-time constant-factor-approximation algorithm for the problem of computing the induced biclique number in a $P_5$-free graph: this time we consider the largest value of~$\ell$ for which the algorithm returns the first outcome. 
The existence of a tree-decomposition with independence number at most~${4\ell}$ is evidence that the induced biclique number of~$G$ is at most~${4\ell}$. 

While the problem of computing the induced biclique number received some attention~\cite{Yannakakis78,DawandeKST01}, we are not aware of any results in $P_5$-free graphs.
We believe that this is an interesting question to investigate.

Finally, let us recall that graphs of bounded tree-independence number have very useful algorithmic properties~\cite{MR4664382,LMMORS24}. 
Furthermore, the complexity of many such algorithms depends only single-exponentially on the tree-independence number~\cite{lokshtanovqpoly}.
Thus, the bounds on the tree-independence number obtained in this paper can be used to obtain quasipolynomial-time (resp., subexponential-time) algorithms for various problems in $P_5$-free graphs of polylogarithmic (resp., sublinear) induced biclique number.

\section*{Acknowledgements}

The project was initiated at the workshop Homonolo 2024. We are grateful to the organizers and other participants, especially Peter Strulo, for a productive atmosphere and inspiring discussions.
We are also grateful to Olivia Milani\v{c} whose patience significantly helped in finishing this paper.

Václav Blažej was supported by the  project BOBR that has received funding from the European Research Council (ERC) under the European Union’s Horizon 2020 research and innovation programme (grant agreement No. 948057). 
J.~Pascal Gollin was supported in part by the Slovenian Research and Innovation Agency (research project N1-0370). 
Tomáš Hons was supported by Project 24-12591M of the Czech Science Foundation (GA\v{C}R).
Tomáš Masařík was supported by the Polish National Science Centre SONATA-17 grant number 2021/43/D/ST6/03312.
Martin Milanič was supported in part by the Slovenian Research and Innovation Agency (I0-0035, research program P1-0285 and research projects J1-60012, J1-70035, J1-70046, and N1-0370).
Paweł Rzążewski was supported by the National Science Centre grant number 2024/54/E/ST6/00094.
Ondřej Suchý was co-funded by the European Union under the project Robotics and advanced industrial production (reg.~no.~CZ.02.01.01/00/22\_008/0004590).
Alexandra Wesolek was supported by the Deutsche Forschungsgemeinschaft (DFG, German Research Foundation) under Germany's Excellence Strategy – The Berlin Mathematics Research Center MATH+ (EXC-2046/1, project ID:390685689).

\printbibliography

\end{document}